\documentclass[a4paper,11pt]{article}
\usepackage[utf8]{inputenc}
\usepackage[english]{babel}

\usepackage{url}
\usepackage{multicol}
\usepackage{amssymb, amsmath, amsthm}
\usepackage{latexsym}
\usepackage{verbatim}
\usepackage{graphicx}  
\usepackage{lscape} 
\input xy
\xyoption{all}

\usepackage{indentfirst} 

\usepackage{verbatim}

\newtheorem{theorem}{Theorem}

\newtheorem{conjecture}[theorem]{Conjecture}

\newtheorem{lemma}[theorem]{Lemma}

\newtheorem{proposition}[theorem]{Proposition}

\newtheorem{definition}{Definition}[section]

\newtheorem{remark}{Remark}[section]

\title{Asymptotics for Magic Squares of Primes}

\author{Carlos Vinuesa}

\date{\today}

\begin{document}

\maketitle

\begin{abstract}
Based on the work of Green, Tao and Ziegler, we give asymptotics when $N \to \infty$ for the number of $n \times n$ magic squares with their entries being prime numbers in $[0,N]$. For every $n \ge 3$ we give appropriate systems of linear forms (or equivalently basis) describing all $n \times n$ magic squares with integer entries and we calculate the complexity of these systems in the Green and Tao sense. We compute the precise asymptotics for the cases $n=3$ (complexity 3) and $n=4$ (complexity 1), and the given algorithm works for $n \ge 5$ (complexity 1). Finally, we show that the asymptotics are exactly the same if we impose that all the entries of the magic squares have to be different.
\end{abstract}

\section{Introduction}\label{intro}

In this article we will be interested in the asymptotic number of magic squares with their entries being prime numbers in $[0,N]$ when $N \to \infty$.

\bigskip

The generalization of an important conjecture by Hardy and Littlewood, says (Conjecture 1.4 in \cite{GT1}):
\begin{conjecture}\label{conjeHL}
(Generalised Hardy-Littlewood conjecture) Let $N$, $d$, $t$, $L$ be positive integers, and let $\Psi:\mathbb{Z}^d \rightarrow \mathbb{Z}^t$, $\Psi = (\psi_1, \ldots, \psi_t)$, be a system of affine-linear forms (none of the $\psi_i$ is constant and no two of them are rational multiples of each other) with size $||\Psi||_N := \sum_{i=1}^t \sum_{j=1}^d |\dot{\psi}_i(e_j)| + \sum_{i=1}^t \left| \frac{\psi_i(0)}{N}\right| \le L$, where each $\psi_i = \dot{\psi_i} + \psi_i(0)$ and $e_1, \ldots, e_d$ is the standard basis for $\mathbb{Z}^d$. Let $K \subset [-N,N]^d$ be a convex body. Then we have:
$$|K \cap \mathbb{Z}^d \cap \Psi^{-1}(P^t)| = (1 + o_{t,d,L}(1)) \frac{\beta_{\infty}}{\log ^t N} \prod_{p \text{ prime}} \beta_p + o_{t,d,L}\left( \frac{N^d}{\log ^t N} \right),$$
where $P = \{2, 3, 5, \ldots \}$ denotes the prime numbers, the archimedean factor $\beta_{\infty}$ is $\text{vol}_d(K \cap \Psi^{-1}(\mathbb{R}_+^t))$ and the local factors $\beta_p$ are defined as $\beta_p = \mathbb{E}_{n \in \mathbb{Z}_p^d} \prod_i \Lambda_{\mathbb{Z}_p} (\psi_i(n))$ for each prime $p$ with $\Lambda_{\mathbb{Z}_p}$ being the \textit{local von Mangoldt function}, $\Lambda_{\mathbb{Z}_p}:\mathbb{Z}_p \rightarrow \mathbb{R}^+$, $\Lambda_{\mathbb{Z}_p}(0) = 0$ and $\Lambda_{\mathbb{Z}_p}(b) = \frac{p}{\phi(p)} = \frac{p}{p-1}$ if $b \neq 0$.
\end{conjecture}

Informally speaking, in the same sense that the Prime Number Theorem states that if a random integer is selected near $N$ its probability of being prime is asymp\-totically  $1/\log N$, the conjecture asserts that the probability that a randomly selected point in $\Psi(\mathbb{Z}^d) \cap \mathbb{Z}_+^t$ ``of magnitude $N$'' has prime entries in all its coordinates is asymptotically $\prod_p \beta_p / \log^t N$.

\bigskip

Green and Tao introduced in \cite{GT1} a ``measure of how complicated such a problem is''. It is called complexity (or Cauchy-Schwarz complexity):
\begin{definition} \label{complexity}
If $\Psi = (\psi_1, \ldots, \psi_t)$ is a system of affine-linear forms, it has $i$-complexity at most $s$ if one can write the set of $t-1$ forms $\{\psi_1, \ldots, \psi_{i-1},\psi_{i+1}, \ldots, \psi_t \}$ as a union of $s+1$ sets, in such a way that $\psi_i$ does not lie in the affine-linear span over $\mathbb{Q}$ of any of these sets. The complexity of $\Psi$ is the least $s$ for which the system has $i$-complexity at most $s$ for all $1 \le i \le t$, or $\infty$ if no such $s$ exists. 
\end{definition}

The Main Theorem of Green and Tao in that same article, says:
\begin{theorem}
Suppose that the Inverse Gowers-norm conjecture $GI(s)$ and the M$\ddot{o}$bius and Nilsequences conjecture $MN(s)$ are true for some finite $s \ge 1$. Then the generalised Hardy-Littlewood conjecture is true for all systems of affine-linear forms of complexity at most $s$.
\end{theorem}

In \cite{GT2} the \textit{M$\ddot{o}$bius and Nilsequences conjecture} $MN(s)$ is proved for every $s$ and in \cite{GT3} the \textit{Inverse Gowers-norm conjecture} $GI(s)$ is proved for every $s$. Then, Conjecture \ref{conjeHL} is in fact a theorem for every system of affine-linear forms of finite complexity.

\bigskip

For the problem of $n \times n$ magic squares, as we shall soon see, we will be dealing with linear instead of affine-linear forms. In Section \ref{sec2basis}, we will find suitable linear forms for our problem for every $n \ge 3$. In Section \ref{seccomp}, we will calculate the complexity of all these systems of linear forms. In Section \ref{sec4vol}, we will deal with the computation of the volumes of the polytopes that arise in our problem. In Section \ref{sec5local}, we will face the computation of the local factors $\beta_p$. Based on the work of the other sections, in Section \ref{sec6asym} we will be able to give our asymptotics. Finally, in Section \ref{sec7dif}, we will show that the asymptotics are the same if we impose the condition that our magic squares must have different entries.

\section{Basis for Magic Squares}\label{sec2basis}

\begin{definition}
For any $n \ge 3$, an $n \times n$ $\mathbb{R}$-magic square is an $n \times n$ array with its $n^2$ entries being real numbers, in such a way that the sum of the elements in every row, column or any of the two main diagonals is the same, the magic sum $S$. If all its entries are integers, we will say it is a $\mathbb{Z}$-magic square.
\end{definition}

\begin{remark}
For every $n \ge 3$, $n \times n$ $\mathbb{R}$-magic squares have the structure of a $\mathbb{R}$-vector space and $n \times n$ $\mathbb{Z}$-magic squares form an abelian group with the sum, and so a $\mathbb{Z}$-module, and we will soon see that we can find a basis for this $\mathbb{Z}$-module. Most of the time in this article we will be only interested in $\mathbb{Z}$-magic squares. When we want to specify that the $n^2$ entries of a magic square are different (in fact, what is usually called a magic square is a $\mathbb{Z}$-magic square with different entries), we will explicitly state it. Also, we will specify when we want the entries to belong to certain set (e.g. $\{0, \ldots , N\}$ or the primes in this set).
\end{remark}

For a $3 \times 3$ $\mathbb{Z}$-magic square, naming its entries $x_1, x_2, \ldots, x_9$ from left to right and from top to bottom (as we will always do in the rest of the article\footnote{We will think about these vectors with $n^2$ coordinates both as column vectors and as $n \times n$ squares. So, we will sometines make an abuse of notation refering to the second row of a vector (entries $n+1, \ldots, 2n$) or its top-left to bottom-right diagonal (entries $1, n+2, 2n+3, \ldots, n^2$).}), we have:
$$(x_1 + x_5 + x_9) + (x_2 + x_5 + x_8) + (x_3 + x_5 + x_7) = 3S$$
and since $x_1 + x_2 + x_3 = x_7 + x_8 + x_9 = S$, we have that $S = 3 x_5$. So, if the integer $x_5 = a$, then $x_1$ and $x_9$ will be respectively $a+b$ and $a-b$ for some integer $b$. In the same way, $x_3$ and $x_7$ will be respectively $a+c$ and $a-c$ for some integer $c$. The other four entries of the square are now determined, as we can see in Figure \ref{figura1}.

\begin{figure}[h]
\begin{center}
\begin{tabular}{l c r c l c r}
  $x_1$ & $x_2$ & $x_3$ &  \qquad \qquad & $a+b$ & $a-b-c$ & $a+c$ \\
  $x_4$ & $x_5$ & $x_6$ & $=$ & $a-b+c$ & $a$ & $a+b-c$ \\
  $x_7$ & $x_8$ & $x_9$ & \qquad \qquad & $a-c$ & $a+b+c$ & $a-b$ \\
\end{tabular}
\end{center}
\caption{Entries of a $3 \times 3$ magic square.}
\label{figura1}
\end{figure}

So, we already have our system of linear forms for $3 \times 3$ $\mathbb{Z}$-magic squares, $\Psi: \mathbb{Z}^3 \to \mathbb{Z}^9$, $\Psi(a,b,c)=(a+b,a-b-c,a+c,a-b+c,a,a+b-c,a-c,a+b+c,a-b)$.

\bigskip

Not only for any $a, b, c \in \mathbb{Z}$ we have that $\Psi(a,b,c)$ gives the entries of a $\mathbb{Z}$-magic square, but also, because of the way we constructed it, we have that every $3 \times 3$ $\mathbb{Z}$-magic square is obtained for some $a, b, c \in \mathbb{Z}$ (and this last thing is important because, if not, our linear forms could be describing only a subset of all $\mathbb{Z}$-magic squares). In other words, and this approach will be useful for magic squares with larger side, looking at the linear forms, we conclude that given any $3 \times 3$ $\mathbb{Z}$-magic square, we can obtain it as an integer linear combination of the three magic squares given in Figure \ref{figure3x3}.

\begin{figure}
\begin{center}
\begin{tabular}{l c r c l c r c l c r}
  $1$ & $1$ & $1$ & \qquad \qquad & $1$ & $-1$ & $0$ & \qquad \qquad & $0$ & $-1$ & $1$ \\
  $1$ & $1$ & $1$ & \qquad \qquad & $-1$ & $0$ & $1$ & \qquad \qquad & $1$ & $0$ & $-1$ \\
  $1$ & $1$ & $1$ & \qquad \qquad & $0$ & $1$ & $-1$ & \qquad \qquad & $-1$ & $1$ & $0$ \\
\end{tabular}
\end{center}
\caption{$\mathbb{Z}$-basis for $3 \times 3$ $\mathbb{Z}$-magic squares.}
\label{figure3x3}
\end{figure}

\bigskip

Then, we have a basis for the $\mathbb{Z}$-module of $3 \times 3$ $\mathbb{Z}$-magic squares. Obtaining this basis is clearly equivalent to obtaining the linear forms we are interested in: if we write the entries of each one of the squares of the basis as column vectors, one after another, and we look at the rows, we will read the coefficients of the linear forms.

\bigskip

Since a finite set of vectors with integer entries is linearly independent over $\mathbb{Z}$ if and only if it is linearly independent over $\mathbb{R}$ (the ``if'' part is obvious and the ``only if'' part can be deduced from the fact that, thinking the vectors of the set as rows of a matrix, elementary row operations over $\mathbb{Z}$ give a row reduced form over $\mathbb{Z}$ of the matrix), in what follows we can interpret if we want ``linearly independent'' as ``linearly independent over $\mathbb{R}$''.

\begin{definition} \label{zbasis}
We will say that a linearly independent set of $n \times n$  $\mathbb{Z}$-magic squares is a $\mathbb{Z}$-basis for $n \times n$  $\mathbb{Z}$-magic squares if every $n \times n$ $\mathbb{Z}$-magic square can be obtained as an integer linear combination of its elements. We will say that a linearly independent set of $n \times n$ $\mathbb{R}$-magic squares is an $\mathbb{R}$-basis for $n \times n$ $\mathbb{R}$-magic squares if every $n \times n$ $\mathbb{R}$-magic square can be obtained as a real linear combination of its elements.
\end{definition}

We remark that the given $\mathbb{Z}$-basis for $3 \times 3$ magic squares (as every other $\mathbb{Z}$-basis) is also an $\mathbb{R}$-basis. But not every $\mathbb{R}$-basis (even if all its elements have only integer entries) is a $\mathbb{Z}$-basis. So, how to obtain an $\mathbb{R}$-basis for $n \times n$ magic squares and how to make sure it is a $\mathbb{Z}$-basis?

\bigskip

\label{sistlineqA}First of all we can easily obtain a system of linear equations that defines $n \times n$ $\mathbb{R}$-magic or $\mathbb{Z}$-magic squares. First, we have $n - 1$ conditions (linear equations) of \textit{type $A$} saying that the sum of the entries of the $i$-th row is equal to the sum of the entries of the $i+1$-th row (for $i = 1, \ldots, n-1$). Then, we have other $n - 1$ conditions of \textit{type $B$} saying that the sum of the entries of the $i$-th column is equal to the sum of the entries of the $i+1$-th column (for $i = 1, \ldots, n-1$). These conditions also imply that the sum of every row is equal to the sum of every column. Now, we have condition $C$ saying that the sum of the entries of one main diagonal is equal to the sum of the entries of the other main diagonal. Finally, we have condition $D$ saying that the sum of the entries of the main diagonal from top-left to bottom-right is equal to the sum of the entries of (say) the first column. These $2n$ conditions clearly define $n \times n$ $\mathbb{R}$-magic or $\mathbb{Z}$-magic squares. If we name $x$ the column vector with entries $x_1, \ldots, x_{n^2}$, and $0$ the column vector with $2n$ zeros, then $n \times n$ $\mathbb{R}$-magic or $\mathbb{Z}$-magic squares are defined by the system of linear equations $Ax = 0$, where $A$\label{matriza} is the matrix formed with the coefficients of the $2n$ conditions previously described, which, for example, for $n=4$ would be the one given in Figure \ref{figure2}.
\begin{figure}
$$ \left( \begin{array}{cccccccccccccccc}
1 & 1 & 1 & 1 & -1 & -1 & -1 & -1 & 0 & 0 & 0 & 0 & 0 & 0 & 0 & 0 \\
0 & 0 & 0 & 0 & 1 & 1 & 1 & 1 & -1 & -1 & -1 & -1 & 0 & 0 & 0 & 0 \\
0 & 0 & 0 & 0 & 0 & 0 & 0 & 0 & 1 & 1 & 1 & 1 & -1 & -1 & -1 & -1 \\
1 & -1 & 0 & 0 & 1 & -1 & 0 & 0 & 1 & -1 & 0 & 0 & 1 & -1 & 0 & 0 \\
0 & 1 & -1 & 0 & 0 & 1 & -1 & 0 & 0 & 1 & -1 & 0 & 0 & 1 & -1 & 0 \\
0 & 0 & 1 & -1 & 0 & 0 & 1 & -1 & 0 & 0 & 1 & -1 & 0 & 0 & 1 & -1 \\
1 & 0 & 0 & -1 & 0 & 1 & -1 & 0 & 0 & -1 & 1 & 0 & -1 & 0 & 0 & 1 \\
0 & 0 & 0 & 0 & 1 & -1 & 0 & 0 & 1 & 0 & -1 & 0 & 1 & 0 & 0 & -1
\end{array} \right) $$
\caption{Matrix $A$ for $n = 4$.}
\label{figure2}
\end{figure}

\bigskip

Are the $2n$ rows of the matrix $A$ linearly independent? Yes, they are. \label{elegante}An elegant way to see it is to exhibit, for every $i = 1, \ldots, 2n$, a vector that satisfies the conditions in all the previous rows but not the condition in row $i$. The vectors of \textit{type $a$} with ones in the first $in$ positions (the first $i$ rows) and zeros in the rest (for $i = 1, \ldots, n-1$) satisfy all \textit{type $A$} conditions except from condition $i$ (in particular all conditions above row $i$ and not condition in row $i$). The same happens with vectors of \textit{type $b$} that have ones in the positions ``$j \ +$ multiple of $n$'' for $j = 1, \ldots, i$ (the first $i$ columns) and zeros in the rest (for $i = 1, \ldots, n-1$): they satisfy all conditions above $n-1+i$ but not condition $n-1+i$. The vector $c$ with one in all its positions except from positions $1, n+2, 2n+3, \ldots, n^2$ (the main diagonal from top-left to bottom-right) where it has zeros, satisfies all conditions of \textit{type $A$} and \textit{type $B$} but not condition $C$. Finally, the vector $d$ with zero in all its positions except from positions $1, n, n+2, 2n-1, 2n-3, 3n-2, \ldots, n^2-n+1, n^2$ (the two main diagonals) where it has ones (or a two in position $(n^2 + 1)/2$ when $n$ is odd), satisfies all conditions of \textit{type $A$} and \textit{type $B$} and condition \textit{C} but not condition $D$. Figure \ref{figure3} shows an example of how vectors of \textit{type $a$} and \textit{type $b$} and vectors $c$ and $d$ look like when we write them as $n \times n$ squares, for $n = 5$.


\begin{figure}
\begin{footnotesize}
\begin{center}
\begin{tabular}{c c c c c c c c c c c c c c c c c c c c c c c c c c c c c}
  $1$ & $1$ & $1$ & $1$ & $1$ & \qquad & $1$ & $1$ & $1$ & $0$ & $0$ & \qquad & $0$ & $1$ & $1$ & $1$ & $1$ & \qquad & $1$ & $0$ & $0$ & $0$ & $1$ \\
  $1$ & $1$ & $1$ & $1$ & $1$ & \qquad & $1$ & $1$ & $1$ & $0$ & $0$ & \qquad & $1$ & $0$ & $1$ & $1$ & $1$ & \qquad & $0$ & $1$ & $0$ & $1$ & $0$ \\
  $0$ & $0$ & $0$ & $0$ & $0$ & \qquad & $1$ & $1$ & $1$ & $0$ & $0$ & \qquad & $1$ & $1$ & $0$ & $1$ & $1$ & \qquad & $0$ & $0$ & $2$ & $0$ & $0$ \\
  $0$ & $0$ & $0$ & $0$ & $0$ & \qquad & $1$ & $1$ & $1$ & $0$ & $0$ & \qquad & $1$ & $1$ & $1$ & $0$ & $1$ & \qquad & $0$ & $1$ & $0$ & $1$ & $0$ \\
  $0$ & $0$ & $0$ & $0$ & $0$ & \qquad & $1$ & $1$ & $1$ & $0$ & $0$ & \qquad & $1$ & $1$ & $1$ & $1$ & $0$ & \qquad & $1$ & $0$ & $0$ & $0$ & $1$ 
\end{tabular}
\end{center}
\end{footnotesize}
\caption{Examples of vectors of  \textit{type $a$} and \textit{type $b$} and vectors $c$ and $d$ for $n=5$.}
\label{figure3}
\end{figure}

\bigskip

We deduce that the rank of matrix $A$ is $2n$, and so the $\mathbb{R}$-vector space of $\mathbb{R}$-magic squares has dimension $n^2 - 2n$. To find an $\mathbb{R}$-basis of this vector space we need to find $n^2 - 2n$ linearly independent vectors that satisfy the $2n$ equations, that is, $2n$ linearly independent magic squares. This can be done easily for every fixed $n$, since we know how to solve a system of linear equations. But we also want our $\mathbb{R}$-basis to be a $\mathbb{Z}$-basis. In principle it would not be obvious even that a $\mathbb{Z}$-basis exists.

\bigskip

We could, for example, look for an appropriate matrix obtained from $A$ by elementary row operations, with $2n$ of its columns being the vectors of the canonical basis of the $2n$-dimensional Euclidean space and with all the rest of its entries being integer numbers. Instead of that, and maybe more naturally, we will give an explicit $\mathbb{Z}$-basis for $n \times n$ magic squares and this will prove in particular that such a matrix exists for every $n \ge 4$. \label{comentariomatrices}

\bigskip

We know the number of linearly independent vectors that we need: $n^2 - 2n$. In the case $n = 3$ we had $3^2 - 2 \cdot 3 = 3$ vectors in our $\mathbb{Z}$-basis. For $n = 4$ we need $4^2 - 2 \cdot 4 = 8$ linearly independent magic squares that form a $\mathbb{Z}$-basis.

\bigskip

The easiest way to understand our basis for $n \times n$ $\mathbb{Z}$-magic squares for $n \ge 4$ is the next. We will carefully select some positions of the square, more concretely $n^2 - 2n$ positions, in such a way that one time we fix the entries in these positions, the rest of the entries of the magic square are determined (we have discovered that this idea is already pointed out in \cite{W}). The selected positions will form the skeleton of the basis (following the notation in \cite{W}). Then, the squares of the basis will be the only $n^2 - 2n$ squares that have one as an entry in one position of the skeleton and zeros in all the rest of it.

\bigskip

For $n = 4$, selecting the next skeleton of 8 positions,
\begin{center}
\includegraphics[width=0.2\textwidth]{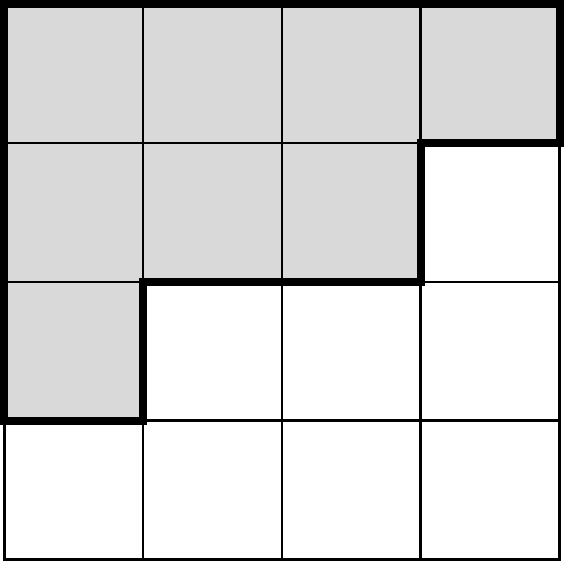}
\end{center}
\noindent we have the $\mathbb{Z}$-basis for $4 \times 4$ $\mathbb{Z}$-magic squares shown in Figure \ref{figure4}.

\begin{figure}[h]
\begin{footnotesize}
\begin{center}
\begin{tabular}{cccccccccccccc}
$\textbf{1}$ & $\textbf{0}$ & $\textbf{0}$ & $\textbf{0}$ & \qquad & $\textbf{0}$ & $\textbf{1}$ & $\textbf{0}$ & $\textbf{0}$ & \qquad & $\textbf{0}$ & $\textbf{0}$ & $\textbf{1}$ & $\textbf{0}$ \\
$\textbf{0}$ & $\textbf{0}$ & $\textbf{0}$ & $1$ & \qquad & $\textbf{0}$ & $\textbf{0}$ & $\textbf{0}$ & $1$ & \qquad & $\textbf{0}$ & $\textbf{0}$ & $\textbf{0}$ & $1$ \\
$\textbf{0}$ & $1$ & $0$ & $0$ & \qquad & $\textbf{0}$ & $0$ & $1$ & $0$ & \qquad & $\textbf{0}$ & $0$ & $1$ & $0$ \\
$0$ & $0$ & $1$ & $0$ & \qquad & $1$ & $0$ & $0$ & $0$ & \qquad & $1$ & $1$ & $-1$ & $0$ \\
\\
$\textbf{0}$ & $\textbf{0}$ & $\textbf{0}$ & $\textbf{1}$ & \qquad & $\textbf{0}$ & $\textbf{0}$ & $\textbf{0}$ & $\textbf{0}$ & \qquad &  $\textbf{0}$ & $\textbf{0}$ & $\textbf{0}$ & $\textbf{0}$ \\
$\textbf{0}$ & $\textbf{0}$ & $\textbf{0}$ & $1$ & \qquad & $\textbf{1}$ & $\textbf{0}$ & $\textbf{0}$ & $-1$ & \qquad & $\textbf{0}$ & $\textbf{1}$ & $\textbf{0}$ & $-1$ \\
$\textbf{0}$ & $-1$ & $2$ & $0$ & \qquad & $\textbf{0}$ & $1$ & $-1$ & $0$ & \qquad & $\textbf{0}$ & $0$ & $-1$ & $-1$ \\
$1$ & $2$ & $-1$ & $-1$ & \qquad & $-1$ & $-1$ & $1$ & $1$ & \qquad & $0$ & $-1$ & $1$ & $0$ \\
\\
& & $\textbf{0}$ & $\textbf{0}$ & $\textbf{0}$ & $\textbf{0}$ & & & $\textbf{0}$ & $\textbf{0}$ & $\textbf{0}$ & $\textbf{0}$ \\
& & $\textbf{0}$ & $\textbf{0}$ & $\textbf{1}$ & $-1$ & & & $\textbf{0}$ & $\textbf{0}$ & $\textbf{0}$ & $0$ \\
& & $\textbf{0}$ & $-1$ & $0$ & $1$ & & & $\textbf{1}$ & $1$ & $-1$ & $-1$ \\
& & $0$ & $1$ & $-1$ & $0$ & & & $-1$ & $-1$ & $1$ & $1$
\end{tabular}
\end{center}
\end{footnotesize}
\caption{$\mathbb{Z}$-basis for $4 \times 4$ $\mathbb{Z}$-magic squares.}
\label{figure4}
\end{figure}

\smallskip

So everything has worked well but, in principle, we have been lucky. There is not an obvious reason why the entries outside the skeleton should be integers. As an example showing that this does not always happen, a $3 \times 3$ $\mathbb{R}$-magic square is characterised by the three entries of its first row, and the only $\mathbb{R}$-magic square starting with one, zero and zero in these three positions is:
\begin{center}
\begin{tabular}{l c r l}
  $1$ & $0$ & $0$ \\
  $-2/3$ & $1/3$ & $4/3$\\
  $2/3$ & $2/3$ & $-1/3$ & .\\
\end{tabular}
\end{center}

\bigskip

With this in mind, we define our basis for $n \ge 5$:

\begin{definition}
\textbf{(Elephant basis)} For $n \ge 5$, consider the skeleton in an $n \times n$ square consisting of all the elements of the first row, all the elements of rows 2 to $n-2$ except from the last one of each row and all the elements of row $n-1$ except from the second, the $(n-2)$-th and the last one. The $n \times n$ elephant basis is formed by the $n^2 - 2n$ $n \times n$ $\mathbb{Z}$-magic squares that have one as an entry in one of the positions of the skeleton and zeros in all its other positions.
\end{definition}

\begin{figure}[h]
\begin{center}
\includegraphics[width=0.3\textwidth]{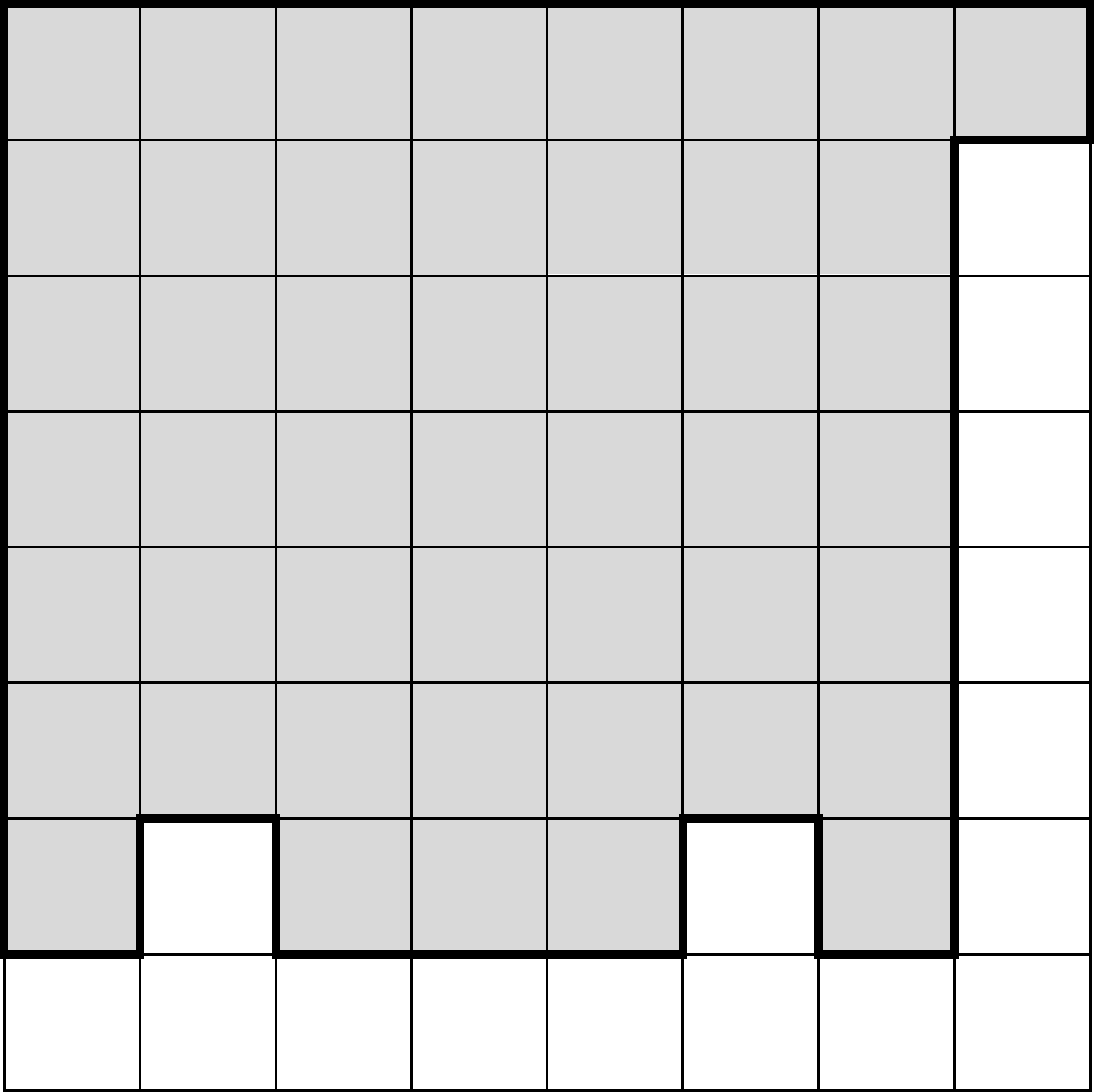}
\end{center}
\caption{Skeleton of the elephant basis.}
\label{elephant}
\end{figure}

To see that the definition is right and that it gives us what we want, we prove the next lemma.

\begin{lemma} \label{lemaele}
For every $n \ge 5$ the elephant basis is a $\mathbb{Z}$-basis for $n \times n$ $\mathbb{Z}$-magic squares. 
\end{lemma}

\begin{proof}
First of all we will show that if we fix any integer entries in the positions of the skeleton (in particular if one of them is 1 and all the rest are 0), they determine a unique $\mathbb{R}$-magic square which is also a $\mathbb{Z}$-magic square. We first observe that the skeleton contains the first row of the square, so adding the values in these positions gives the magic sum of the square. Since the magic sum is an integer number, we can now determine the integer numbers of the last position of rows 2 to $n-2$ (because of the sum on their rows), the last position of columns 1, 3 to $n-3$ and $n-1$ (because of the sum on their columns) and the one in the bottom-right corner of the square (because of the sum on its main diagonal). With these integers already determined, we can determine the integer in position 2 of row $n-1$ (because of the sum on its main diagonal) and the one in the last position of the same row (because of the sum on its column). Now, we can determine the integer in the second position of the last row (because of the sum on its column) and the one in position $n-2$ of row $n-1$ (because of the sum of its row). The value in position $n-2$ of the last row is the only one left to be determined. Since the sum of all the elements of the square is the same if we sum them by rows or by columns, we will of course obtain the same integer if we determine it with the values of column $n-2$ or the ones in the last row.

Now, the vectors are linearly independent by construction (each of them has a 1 in a position where all the others have a 0), and, given a $\mathbb{Z}$-magic square, we can look at the values it has in the positions of the skeleton and we can express it as the linear combination with these coefficients in the corresponding vectors of the elephant basis. So the elephant basis is a $\mathbb{Z}$-basis.
\end{proof}

Following exactly the procedure described in the proof of Lemma \ref{lemaele} we can determine all the entries of all the squares in our elephant basis. If we write them as column vectors, one after another, forming a matrix, the rows will give us the linear forms we were looking for. Figure \ref{matelebasis} gives this $n^2 \times (n^2-2n)$ matrix. We will identify each row of coefficients with the corresponding linear form. The $n^2 - 2n$ forms marked in gray in the figure are the ``canonical'' linear forms with all their coefficients equal to zero except from one which is 1 and correspond with the positions of the skeleton of the elephant basis. We will call these the trivial linear forms. The other $2n$ forms will be called the nontrivial linear forms.

\bigskip

Of course, the first $n$ vectors (columns) have magic sum 1 and all the others have magic sum 0. The pattern on each row, if we look at the blocks of $n-1$ columns indicated (from column $n+1$ to column $n^2 - 3n + 3$), should be obvious. Maybe only some words about the patterns on the third and $(n+3)$-rd row from the bottom should be said. We explain what happens in the third row from the bottom (the other is very similar). The first $n$ and the last $n-3$ positions do not give any problem, so we look at the $n-3$ blocks of $n-1$ positions from column $n+1$ to column $n^2 - 3n + 3$. To obtain the second block from the first (observe that in the case $n = 5$ the first block is $2\ 2\ 0\ 0$), we substract 1 in the second position and add 1 in the third and also substract 1 in the $(n-2)$-nd and add 1 in the $(n-1)$-st. To obtain the third block from the second, the $-1$ and $+1$ on the left move one position to the right (we substract 1 in the third position and add 1 in the fourth) and the $-1$ and $+1$ on the right move one position to the left (we substract 1 in the $(n-3)$-rd position and add 1 in the $(n-2)$-nd). And we continue in the same fashion.

\begin{landscape}
\begin{figure}
\begin{center}
\includegraphics[width=1.4\textwidth]{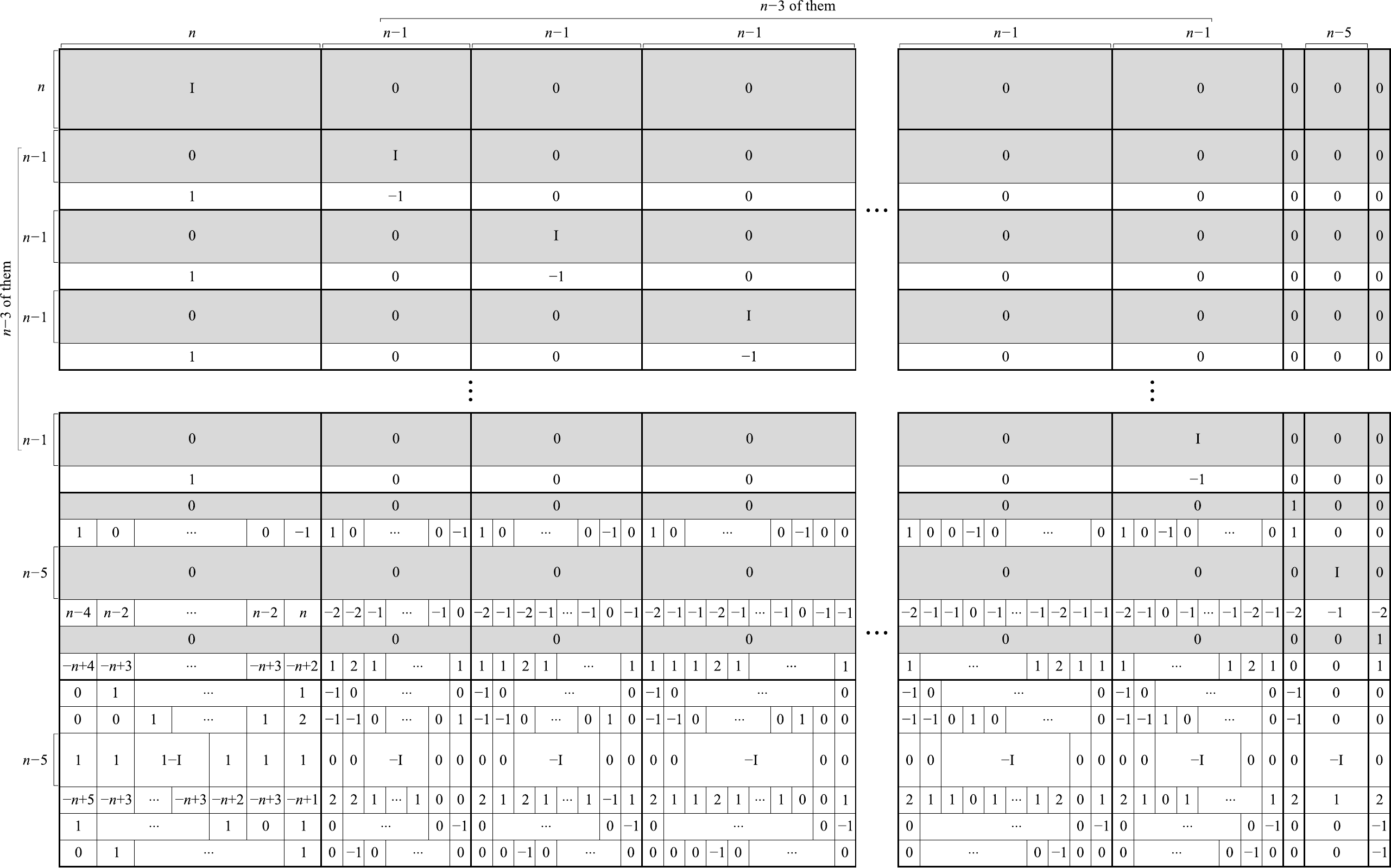}
\end{center}
\caption{Matrix formed with the vectors of the elephant basis as columns (coefficients of linear forms as rows) for any $n \ge 5$.}
\label{matelebasis}
\end{figure}
\end{landscape}

%

As we announced in our comment on page \pageref{comentariomatrices}, obtaining $n^2 - 2n$ trivial linear forms proves in particular that there is a matrix obtained from $A$ by elementary row operations, with $2n$ of its columns being the vectors of the canonical basis of the $2n$-dimensional space and with all the rest of its entries being integer numbers. In fact, we can very easily write such a matrix explicitly from the linear forms: the vectors of the canonical basis of the $2n$-dimensional space will be in the $2n$ columns of the matrix corresponding to the nontrivial linear forms and the other $n^2-2n$ entries of the rows of $A$ will be the opposites of the coefficients of the nontrivial linear forms.


\bigskip

Then we have our linear forms for every $n$. Now what? We turn to the complexity problem.

\section{Complexity}\label{seccomp}
Althought the results mentioned in Section \ref{intro} tell us that we can use Conjecture \ref{conjeHL} as a theorem for every system of affine-linear forms of finite complexity, determining the exact complexity of a problem is theoretically important. Problems of complexity 1 were essentially present in the work of Hardy-Littlewood and Vinogradov, though not in this language. To solve problems of complexity 2 we need to use the results in \cite{GT}, \cite{GT0} and \cite{GT1}. To solve problems of complexity 3 or higher we need to use the results in \cite{GT1}, \cite{GT2} and \cite{GT3}.

\bigskip

We have $\mathbb{Z}$-basis (or equivalently linear forms) for $3 \times 3$, $4 \times 4$ and $n \times n$ (for $n \ge 5$)  $\mathbb{Z}$-magic squares in Figure \ref{figure3x3}, Figure \ref{figure4} and Figure \ref{matelebasis} respectively. We will calculate the complexity (see Definition \ref{complexity}) of these systems of linear forms.

\subsection{Complexity for 3 $\times$ 3 $\mathbb{Z}$-magic squares}

The system of linear forms given by the $\mathbb{Z}$-basis of Figure \ref{figure3x3} is, as we pointed out, $\Psi(a,b,c)=(a+b,a-b-c,a+c,a-b+c,a,a+b-c,a-c,a+b+c,a-b)$. We will indentify these linear forms with the row vectors $(1,1,0)$, $(1,-1,-1)$, $(1,0,1)$, $(1,-1,1)$, $(1,0,0)$, $(1,1,-1)$, $(1,0,-1)$, $(1,1,1)$ and $(1,-1,0)$.

\begin{lemma}
The complexity of the system of linear forms given by the $\mathbb{Z}$-basis for $3 \times 3$ $\mathbb{Z}$-magic squares of Figure \ref{figure3x3} is 3.
\end{lemma}

\begin{proof}
First of all, we show that the $i$-complexity of every form is at most 3. In order to see it, given the form $i$, we split the others into four sets in such a way that the form $i$ cannot be written as a linear combination over $\mathbb{Q}$ of the forms in any of the four sets. We write explicitly the four sets for every form:
\begin{itemize}
\item $(1,1,0)$-complexity at most 3: $\{(1,0,0),(1,0,1)\}$, $\{(1,-1,0),(1,0,-1)\}$,

$\{(1,1,1),(1,-1,1)\}$, $\{(1,1,-1),(1,-1,-1)\}$.
\item $(1,-1,-1)$-complexity at most 3: $\{(1,1,0),(1,-1,0)\}$, $\{(1,0,1),(1,0,-1)\}$,

$\{(1,0,0),(1,-1,1)\}$, $\{(1,1,-1),(1,1,1)\}$.
\item $(1,0,1)$-complexity at most 3: $\{(1,0,0),(1,1,0)\}$, $\{(1,-1,0),(1,0,-1)\}$,

$\{(1,1,1),(1,-1,-1)\}$, $\{(1,-1,1),(1,1,-1)\}$.
\item $(1,-1,1)$-complexity at most 3: $\{(1,1,0),(1,-1,0)\}$, $\{(1,0,1),(1,0,-1)\}$,

$\{(1,1,1),(1,1,-1)\}$, $\{(1,-1,-1),(1,0,0)\}$.
\item $(1,0,0)$-complexity at most 3: $\{(1,1,0),(1,0,1)\}$, $\{(1,-1,0),(1,0,-1)\}$,

$\{(1,1,1),(1,-1,1)\}$, $\{(1,1,-1),(1,-1,-1)\}$.
\item $(1,1,-1)$-complexity at most 3: $\{(1,1,0),(1,-1,0)\}$, $\{(1,0,1),(1,0,-1)\}$,

$\{(1,1,1),(1,-1,1)\}$, $\{(1,-1,-1),(1,0,0)\}$.
\item $(1,0,-1)$-complexity at most 3: $\{(1,0,0),(1,1,0)\}$, $\{(1,-1,0),(1,0,1)\}$,

$\{(1,1,1),(1,-1,-1)\}$, $\{(1,-1,1),(1,1,-1)\}$.
\item $(1,1,1)$-complexity at most 3: $\{(1,1,0),(1,-1,0)\}$, $\{(1,0,1),(1,0,-1)\}$,

$\{(1,0,0),(1,-1,1)\}$, $\{(1,1,-1),(1,-1,-1)\}$.
\item $(1,-1,0)$-complexity at most 3: $\{(1,0,0),(1,0,1)\}$, $\{(1,1,0),(1,0,-1)\}$,

$\{(1,1,1),(1,-1,1)\}$, $\{(1,1,-1),(1,-1,-1)\}$.
\end{itemize}

\bigskip

So the complexity of the system is at most 3. Now we will prove that, for example, the $(1,0,0)$-complexity has to be 3 and no less and this will prove that the complexity of the system is 3. We suppose for a contradiction that we can split the other eight forms in three sets in such a way that $(1,0,0)$ is not in the linear span of any of these sets. Since $(1,0,0)$ is in the linear span over $\mathbb{Q}$ of $(1,0,1)$ and $(1,0,-1)$, these two forms have to be in a different set, and the same happens with $(1,1,0)$ and $(1,-1,0)$. Since we only have three sets, two of these four forms have to be in the same set, so we have that one form of the shape $(1,*,0)$ and one of the shape $(1,0,*)$ are in the same set, where $*$ can be $1$ or $-1$. In that set we cannot have any other form, since $(1,0,0)$ would be in the linear span of them together with any $(1,*,*)$. So the first of our three sets is of the form $\{(1,*,0),(1,0,*)\}$. Now, since $(1,0,0)$ is in the linear span over $\mathbb{Q}$ of $(1,1,1)$ and $(1,-1,-1)$, these two forms have to be in a different set, and the same happens with $(1,-1,1)$ and $(1,1,-1)$. Then, the two forms of the shape $(1,*,0)$ and $(1,0,*)$ that are not in our first set, must be in different sets, since if they are in the same set they would be the only two forms in that set and then we would have more than three sets. Then a form of the shape $(1,*,0)$ is in the second set and a form of the shape $(1,0,*)$ is in the third. Now, either $(1,1,1)$ and $(1,-1,1)$ are in one of these sets and $(1,-1,-1)$ and $(1,1,-1)$ in the other or $(1,1,1)$ and $(1,1,-1)$ are in one these sets and $(1,-1,-1)$ and $(1,-1,1)$ in the other. In the first case, $(0,1,0)$ is in the linear span of both sets, and $(1,0,0)$ in the linear span of the set with the form $(1,*,0)$, which is a contradiction. In the second case, $(0,0,1)$ is in the linear span of both sets, and $(1,0,0)$ in the linear span of the set with the form $(1,0,*)$, which is also a contradiction.
\end{proof}

\bigskip

This tell us two things, that complexity can be somewhat complicated and that for $3 \times 3$ squares we need very ``strong'' and recent results.

\subsection{Complexity for 4 $\times$ 4 $\mathbb{Z}$-magic squares}

We write the vectors of the $\mathbb{Z}$-basis of Figure \ref{figure4} as columns of a matrix (see Figure \ref{matele4}) and, looking at the rows, we have the coefficients of the 16 linear forms, $\psi_1, \ldots, \psi_{16}$, we are interested in. We identify each form with its corresponding row vector in $\mathbb{Z}^8$. Because of the way we constructed the basis, observe that $\psi_1, \ldots, \psi_7$ and $\psi_9$ are the vectors of the canonical basis of the 8-dimensional space.

\begin{figure}[ht]
\begin{center}
\scalebox{0.7}{
$ \left( \begin{array}{cccccccc}
1 & 0 & 0 & 0 & 0 & 0 & 0 & 0 \\
0 & 1 & 0 & 0 & 0 & 0 & 0 & 0 \\
0 & 0 & 1 & 0 & 0 & 0 & 0 & 0 \\
0 & 0 & 0 & 1 & 0 & 0 & 0 & 0\\
0 & 0 & 0 & 0 & 1 & 0 & 0 & 0\\
0 & 0 & 0 & 0 & 0 & 1 & 0 & 0\\
0 & 0 & 0 & 0 & 0 & 0 & 1 & 0\\
1 & 1 & 1 & 1 & -1 & -1 & -1 & 0\\
0 & 0 & 0 & 0 & 0 & 0 & 0 & 1\\
1 & 0 & 0 & -1 & 1 & 0 & -1 & 1\\
0 & 1 & 1 & 2 & -1 & -1 & 0 & -1\\
0 & 0 & 0 & 0 & 0 & 1 & 1 & -1\\
0 & 1 & 1 & 1 & -1 & 0 & 0 & -1\\
0 & 0 & 1 & 2 & -1 & -1 & 1 & -1\\
1 & 0 & -1 & -1 & 1 & 1 & -1 & 1\\
0 & 0 & 0 & -1 & 1 & 0 & 0 & 1
\end{array} \right) $
}
\end{center}
\caption{Matrix with the vectors of the elephant basis as columns for $n = 4$.}
\label{matele4}
\end{figure}

We analyse the $\psi_i$-complexity of each form:
\begin{itemize}
\item $\psi_1$, $\psi_8$, $\psi_{10}$ and $\psi_{15}$-complexities are at most 1 because $\psi_1$, $\psi_8$, $\psi_{10}$ and $\psi_{15}$ are linearly independent and not in the linear span of the rest (all of them have a value different from 0 in their first position and all the rest have a zero in that position).
\item $\psi_2$, $\psi_3$, $\psi_{11}$, $\psi_{13}$ and $\psi_{14}$-complexities are at most 1 because $\psi_2$, $\psi_3$, $\psi_8$, $\psi_{11}$, $\psi_{13}$, $\psi_{14}$ and $\psi_{15}$ are linearly independent and not in the linear span of the rest (all of them have a value different from 0 in their second or third positions and all the rest have a zero in both of these positions).
\item $\psi_{4}$ and $\psi_{16}$-complexities are at most 1 because $\psi_4$, $\psi_8$, $\psi_{10}$, $\psi_{11}$, $\psi_{13}$, $\psi_{14}$, $\psi_{15}$ and $\psi_{16}$ are linearly independent and not in the linear span of the rest (all of them have a value different from 0 in their fourth position and all the rest have a zero in that position).
\item $\psi_{5}$ and $\psi_{9}$-complexities are at most 1 because $\psi_5$, $\psi_6$, $\psi_7$, $\psi_9$, $\psi_{11}$, $\psi_{13}$, $\psi_{14}$ and $\psi_{15}$ are linearly independent and $\psi_{5}$ and $\psi_{9}$ are not in the linear span of the other eight forms.
\item $\psi_{6}$ and $\psi_{12}$-complexities are at most 1 because $\psi_6$, $\psi_8$, $\psi_{11}$, $\psi_{12}$, $\psi_{14}$ and $\psi_{15}$ are linearly independent and not in the linear span of the rest (all of them have a value different from 0 in their sixth position and all the rest have a zero in that position).
\item $\psi_{7}$-complexity is at most 1 because $\psi_7$, $\psi_8$, $\psi_{10}$, $\psi_{12}$, $\psi_{14}$ and $\psi_{15}$ are linearly independent and not in the linear span of the rest (all of them have a value different from 0 in their seventh position and all the rest have a zero in that position).
\end{itemize}

So the $\psi_i$-complexity of every form is at most 1 (and it cannot be less since the 16 linear forms are not, of course, linearly independent). Then:

\begin{lemma}
The complexity of the system of linear forms given by the $\mathbb{Z}$-basis for $4 \times 4$ $\mathbb{Z}$-magic squares of Figure \ref{figure4} is 1.
\end{lemma}

\subsection{Complexity for $n \times n$ $\mathbb{Z}$-magic squares for $n \ge 5$}
Recall that the ``white'' rows in Figure \ref{matelebasis} give the coefficients of the nontrivial linear forms of the elephant basis for every $n \ge 5$.

\begin{lemma}\label{lemacoefnont}
By elementary row operations in the $2n \times (n^2 - 2n)$ matrix with the coefficients of the nontrivial linear forms of the elephant basis as rows (matrix formed by the ``white'' rows in Figure \ref{matelebasis}), the matrix of Figure \ref{nonteche} can be obtained.
\end{lemma}

\begin{proof}
The rows of the matrix
\begin{figure}[h]
\begin{center}
\includegraphics[width=0.5\textwidth]{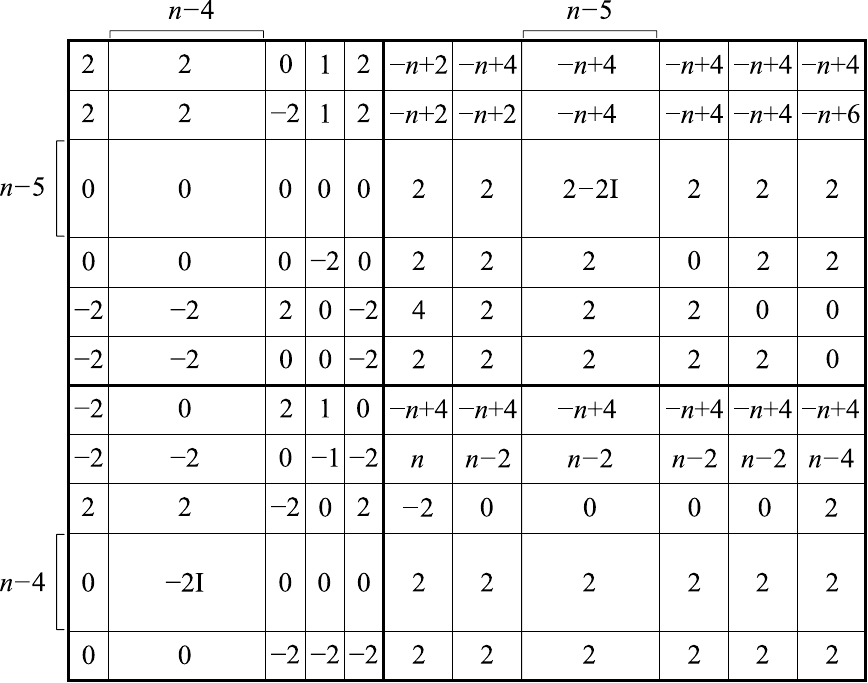}
\end{center}
\label{coef}
\end{figure}

\noindent list the coefficients of the linear combination of the rows in the matrix of coefficients of the nontrivial linear forms that give the rows of the matrix in Figure \ref{nonteche}.
\end{proof}

\begin{proposition}
The complexity of the system of linear forms that define $n \times n$ $\mathbb{Z}$-magic squares given by the elephant basis for $n \ge 5$ is 1. 
\end{proposition}

\begin{proof}
From the previous result, and looking at Figure \ref{nonteche}, we know that the $2n$ nontrivial linear forms of the elephant basis are linearly independent (the $2n$ columns marked in gray have zeros in every position except from one).

\bigskip

Given any trivial linear form different from the $n$-th one, it is impossible to write it as a linear combination of the rows of the matrix in Figure \ref{nonteche} (and so as a linear combination of the nontrivial forms): If such a linear combination exists and we look at the positions corresponding to the columns in grey in Figure \ref{nonteche} at most one of them will be different from zero. This implies that all the coefficients of the linear combination except from at most one are zero. Now, if the trivial linear form is different from the $n$-th one, it cannot be obtained as a linear combination with only one nonzero coefficient of the rows in the matrix of Figure \ref{nonteche} because none of them is proportional to a trivial linear form.

\begin{landscape}
\begin{figure}
\begin{center}
\includegraphics[width=1.4\textwidth]{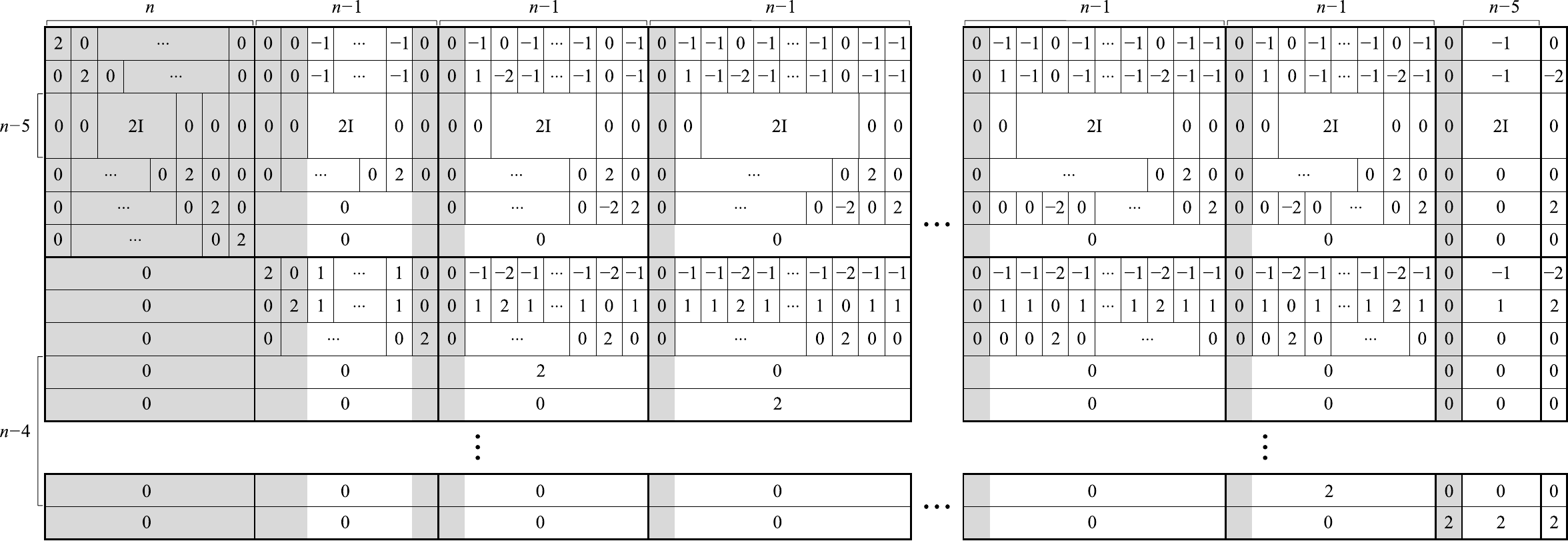}
\end{center}
\caption{Row reduced form of the matrix with the coefficients of the nontrivial linear forms of the elephant basis as rows.}
\label{nonteche}
\end{figure}
\end{landscape}

If we join the first and third trivial linear forms ($(1,0,\ldots,0)$ and $(0,0,1,0,\ldots,0)$) to the nontrivial linear forms, we have a set of $2n + 2$ linearly independent forms: The trivial form $(0,0,1,0,\ldots,0)$, say, is linearly independent with all the nontrivial ones by the last paragraph argument. Now, suppose for a contradiction that we had a linear combination of this linear form and all the nontrivial ones that gave $(1,0,\ldots,0)$ as a result. Then we would also have a linear combination of $(0,0,1,0,\ldots,0)$ and the rows of the matrix of Figure \ref{nonteche} giving $(1,0,\ldots,0)$. The coefficients of this linear combination would be all zero except from, at most, the ones of the first and the third row of the matrix of Figure \ref{nonteche} and the one of $(0,0,1,0,\ldots,0)$. The coefficient of the first row would have to be different from zero. But this row has nonzero entries in positions were the third row and $(0,0,1,0,\ldots,0)$ have zero entries. Then, the linear combination giving $(1,0,\ldots,0)$ is not possible.

\bigskip

Also, from the fact that the nontrivial linear forms together with $(1,0,\ldots,0)$ and $(0,0,1,0,\ldots,0)$ are linearly independent, we can deduce that there is only one possible linear combination of them giving the $n$-th trivial linear form (the coefficients of $(1,0,\ldots,0)$ and $(0,0,1,0,\ldots,0)$ are 0 and the rest are given by the $n$-th row of the figure of the proof of Lemma \ref{lemacoefnont}). Then, if we take out of this set of forms the first nontrivial one (which had a nonzero coefficient in that linear combination), it will not be possible to obtain the $n$-th trivial form as a linear combination.

\bigskip

Finally, none of the nontrivial forms can be obtained as a linear combination of all but the first and third linear forms, because every nontrivial linear form has a nonzero entry either in the first or in the third position.

\bigskip

These things prove that:
\begin{itemize}
\item Given any trivial linear form different from the $n$-th one, its complexity is at most one because we can divide the rest in the set of trivial and the set of nontrivial ones and it will not be in the linear span of any of these sets.

\item Given the $n$-th trivial form, its complexity is at most one because we can construct the next two sets with the rest of the forms: the first one with the first and third trivial forms and with all the nontrivial forms except from the first one; the second one with the rest of the linear forms.

\item Given any nontrivial linear form, its complexity is at most one because we can divide the rest of the forms in the next two sets: the first one with the first and third trivial forms and with all the other nontrivial forms; the second one with the rest of the trivial forms. \qedhere
\end{itemize}
\end{proof}

\section{Volumes of polytopes}\label{sec4vol}
We already have systems of linear forms $\Psi: \mathbb{Z}^{n^2-2n} \rightarrow \mathbb{Z}^{n^2}$ for every $n$ and we know their complexity. Now, we have to choose a suitable convex body $K \subset [-N,N]^{n^2-2n}$ in Conjecture \ref{conjeHL}. Remember that for our problem we know that Conjecture \ref{conjeHL} is a theorem, since we have finite complexity for every $n \ge 3$. Recall that we are interested in counting the number of $n \times n$ $\mathbb{Z}$-magic squares with their entries being primes in $[0,N]$.

\bigskip

Then, for each $n \ge 3$ and $N \ge 0$, we will have $K = K_n(N) = \{x \in \mathbb{R}^{n^2-2n} : 0 \le \psi_i(x) \le N \text{ for }  i = 1, \ldots, n^2\}$, where the coefficients of $\psi_i$ are given in figures \ref{figure3x3} (for $n=3$), \ref{figure4} (for $n=4$) and \ref{matelebasis} (for $n\ge 5$), and the term $|K \cap \mathbb{Z}^{n^2-2n} \cap \Psi^{-1}(P^{n^2})|$ in Conjecture \ref{conjeHL} will be the quantity we are interested in. The first observation is that $K$ is convex since it is the intersection of $n^2$ convex sets.

\bigskip

We need to calculate the archimedean factor $\beta_{\infty, n} = \text{vol}_{n^2-2n}(K \cap \Psi^{-1}(\mathbb{R}_+^{n^2}))$. Since all the points in our $K$ belong also to $\Psi^{-1}(\mathbb{R}_+^{n^2})$, we have $\beta_{\infty, n} = \text{vol}_{n^2-2n}(K)$. Since $\Psi$ is a system of linear forms, the volume of $K_n(N)$ will be equal to $N^{n^2-2n}$ times the volume of the polytope defined by $K_n(1) = \{x \in \mathbb{R}^d : 0 \le \psi_i(x) \le 1 \text{ for }  i = 1, \ldots, n^2\}$. So, we only have to compute $K_n(1)$ for $n \ge 3$.

\bigskip

Observe that $K_n(1)$ is important in its own right. $\beta_{\infty, n} (N) = \text{vol}_{n^2-2n}(K_n(N))=N^{n^2-2n}\text{vol}_{n^2-2n}(K_n(1))$ is the volume of a convex polytope in $\mathbb{R}^{n^2-2n}$ which points with integer coordinates are in bijective correspondence with all the $n \times n$ $\mathbb{Z}$-magic squares with entries in $[0,N]$. For large $N$, the volume of $K_n(N)$ will be asymptotically the same as the number of integer points in it, so the volume of this polytope gives an asymptotic for the number of $n \times n$ $\mathbb{Z}$-magic squares with entries in $[0, N]$ when $N$ tends to infinity. We will talk about this in more detail in a moment.

\subsection{Volume for 3 $\times$ 3 $\mathbb{Z}$-magic squares}
We have to calculate the volume of the polytope $K_3(1)$ defined by the next 18 inequalities ($0 \le \psi_i(x) \le 1$ for $i = 1, \ldots, 9$): 
$0 \le a+b \le 1$, $0 \le a-b-c \le 1$, $0 \le a+c \le 1$, $0 \le a-b+c \le 1$, $0 \le a \le 1$, $0 \le a+b-c \le 1$, $0 \le a-c \le 1$, $0 \le a+b+c \le 1$, $0 \le a-b \le 1$. Each one of these pairs of inequalities defines a region bounded by two parallel planes and because we are in a tridimensional space we can ``see'' the polytope, which is the one of Figure \ref{tetraedro}. 

\bigskip

\begin{figure}
\begin{center}
\includegraphics[width=0.6\textwidth]{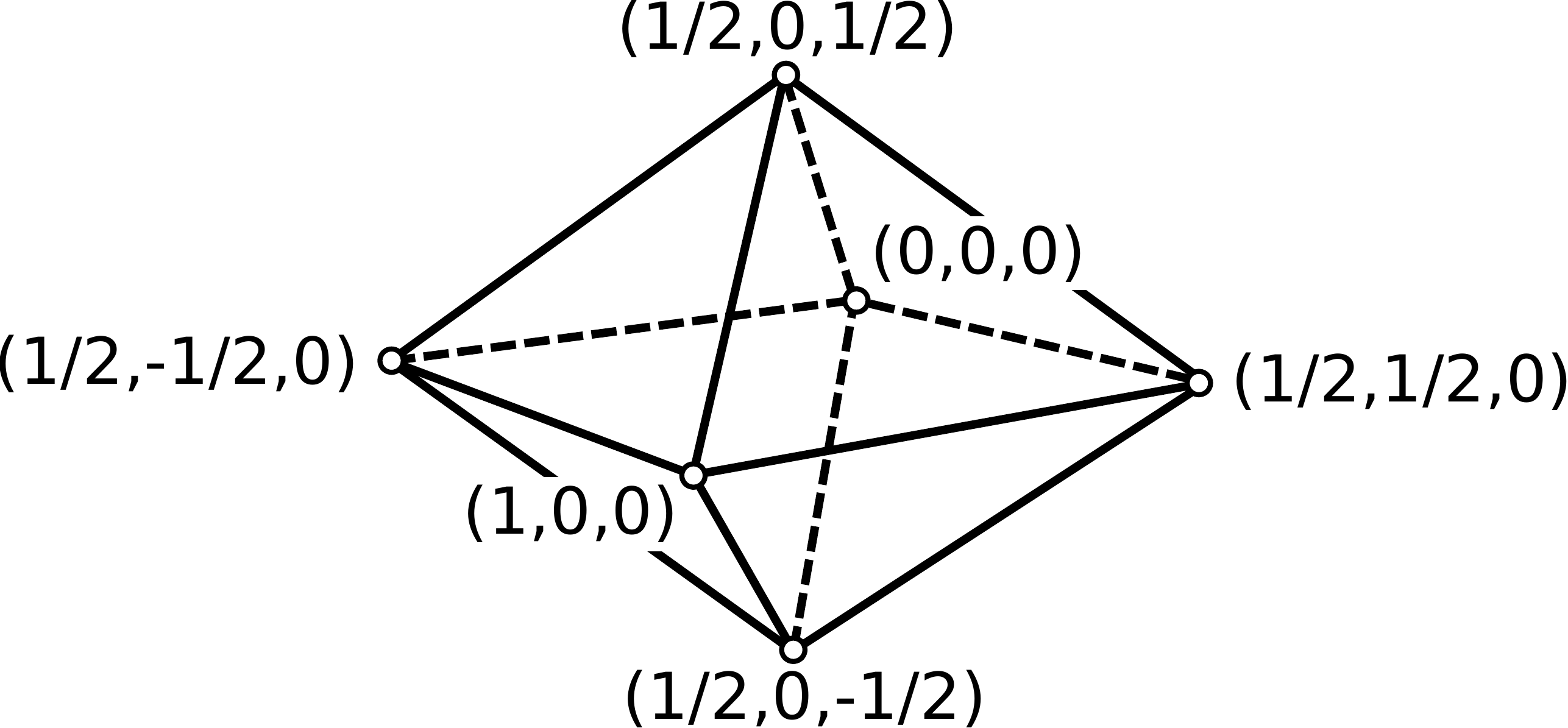}
\end{center}
\caption{Polytope $K_3(1)$.}
\label{tetraedro}
\end{figure}

Since the polytope is the disjoint union of two pyramids with area of their base equal to 1/2 and height 1/2, each pyramid will have volume 1/12 and the polytope will have volume 1/6.

\bigskip

So, $\beta_{\infty, 3} (N)= \frac{N^3}{6}$.

\bigskip

Before jumping to the case $n = 4$, we could think of a different approach to calculate the volume of this polytope, maybe one wich gives a better understanding or is more generalizable. An interesting thing to know will be the exact number of points with integer coordinates in $K = K_n(N)$ for every $N$. As we were discussing before, this will be the exact number of $n \times n$ $\mathbb{Z}$-magic squares with entries in $[0,N]$. We have to go to Ehrhart's theory, and look at its fundamental result (see, for example, \cite{BS} or \cite{E}):

\begin{theorem}
\emph{\textbf{(Ehrhart's Theorem)}} If $P$ is a convex polytope whose vertices have rational coordinates, then the number of points with integer coordinates in the dilations $NP$ for $N = 0, 1, 2 \ldots$ is a quasi-polynomial in $N$ whose degree is the dimension of $P$ and whose period divides the least common multiple of the denominators of the vertices of $P$.
\end{theorem}

In this case, since the least common multiple of the denominators of the vertices of $K_3(1)$ is 2, we will have that the Ehrhart quasipolymial has degree 3 and period at most 2. A direct way to obtain the quasipolynomial, $E_3(N)$, would be to interpolate it. For that we would need $(3 + 1) 2 = 8$ values. $E_3(N)$ is the number of points with integer coordinates, $x$, satisfying $0 \le \psi_i(x) \le N$ for $i = 1, \ldots, 9$ and, with a computer (or even by hand!), it is easy to compute:
\begin{displaymath}
\begin{array}{llll}
E_3(0) = 1 & \quad E_3(1) = 2 & \quad E_3(2) = 7 & \quad E_3(3) = 12\\
E_3(4) = 25 & \quad E_3(5) = 38 & \quad E_3(6) = 63 & \quad E_3(7) = 88
\end{array}
\end{displaymath}

Now, we can interpolate our quasipolynomial
\begin{displaymath}
   E_3(N) = \left\{
     \begin{array}{lr}
       N^3/6 + N^2/2 + 4N/3 + 1 & \quad \text{if} \quad N \equiv 0 \pmod{2}\\
       N^3/6 + N^2/2 + 5N/6 + 1/2 & \quad \text{if} \quad N \equiv 1 \pmod{2}
     \end{array}
   \right.
\end{displaymath} 
which gives the exact number of points with integer coordinates in $K_3(N)$, i. e. the exact number of $3 \times 3$ $\mathbb{Z}$-magic squares with entries in $[0,N]$. The coefficient of the leading term is $1/6$, exactly the volume of $K_3(1)$. This is not a coincidence, and it is a general fact that the leading coefficient of the Ehrhart quasipolynomial of a polytope is its volume. It is an inmediate consequence of the observation that the volume of a $d$-dimensional polytope can be thought as the limit when $N \rightarrow \infty$ of the number of integer points in its $N$-dilation divided by $N^d$. With all this in mind, we can go to the next case. 

\subsection{Volume for 4 $\times$ 4 $\mathbb{Z}$-magic squares} \label{vol4x4}
Our approach is now clear. We want to calculate the volume of the polytope $K_4(1)$ given by the 16 pairs of inequalities given by $0 \le \psi_i(x) \le 1$ for $i = 1, \ldots, 16$, where the coefficients of $\psi_i(x)$ are given by the $i$-th row of the matrix of Figure \ref{matele4}. We know that the volume is given by the leading coefficient of the quasipolynomial, $E_4(N)$, that gives the exact number of integer points satisfying $0 \le \psi_i(x) \le N$ for $i = 1, \ldots, 16$.

\bigskip

We will compute $E_4(N)$ by interpolation. The period of $E_4(N)$ divides the least common multiple of the denominators of the vertices of $K_4(N)$. So, in order to know how many values of $E_4(N)$ we need to interpolate it, we have to calculate the vertices of the polytope. There are different methods to compute the vertices of a polytope (see, for example, \cite{MR} or \cite{D} for surveys of vertex finding algorithms). The most direct of them would be to look at all the intersections of 8 of the 32 equalities  $\psi_i(x) = 0$ or $\psi_i(x) = 1$ for $i = 1, \ldots, 16$ consisting of a single point and then checking if the point satisfies all the inequalities $0 \le \psi_i(x) \le 1$ for $i = 1, \ldots, 16$. You do not want to do this by hand, but your computer will tell you that the 178 vertices of $K_4(1)$ are the ones listed in Appendix \ref{apena}. The only important thing for us is that the denominators of the coordinates of the vertices are one, two or three, so their least common multiple is 6. Then we are looking for a quasipolynomial, $E_4(N)$, of degree 8 and period dividing 6. In order to interpolate $E_4(N)$ we would then need $(8 + 1) 6 = 54$ values.

\bigskip

In Appendix \ref{apenb} we calculate the 54 needed values. With them we interpolate our quasipolynomial and we have

\bigskip

\scalebox{0.81}{
$   E_4(N) = \left\{
     \begin{array}{l}
\frac{8389 N^8}{120960} + \frac{8389 N^7}{15120} + \frac{5531 N^6}{2592} + \frac{10877 N^5}{2160} + \frac{8663 N^4}{1080} + \frac{14371 N^3}{1620} + \frac{143 N^2}{21} + \frac{1067 N}{315} + 1 \\
\text{if} \quad N \equiv 0 \pmod{6} \\
\\
\frac{8389 N^8}{120960} + \frac{8389 N^7}{15120} + \frac{5531 N^6}{2592} + \frac{10877 N^5}{2160} + \frac{69169 N^4}{8640} + \frac{57079 N^3}{6480} + \frac{4303 N^2}{672} + \frac{13607 N}{5040} + \frac{37}{128} \\
\text{if} \quad N \equiv 1 \pmod{6} \\
\\
\frac{8389 N^8}{120960} + \frac{8389 N^7}{15120} + \frac{5531 N^6}{2592} + \frac{10877 N^5}{2160} + \frac{8663 N^4}{1080} + \frac{14371 N^3}{1620} + \frac{143 N^2}{21} + \frac{1067 N}{315} + \frac{97}{81}\\
\text{if} \quad N \equiv 2 \pmod{6} \\
\\
\frac{8389 N^8}{120960} + \frac{8389 N^7}{15120} + \frac{5531 N^6}{2592} + \frac{10877 N^5}{2160} + \frac{69169 N^4}{8640} + \frac{57079 N^3}{6480} + \frac{4303 N^2}{672} + \frac{13607 N}{5040} + \frac{37}{128} \\
\text{if} \quad N \equiv 3 \pmod{6} \\
\\
\frac{8389 N^8}{120960} + \frac{8389 N^7}{15120} + \frac{5531 N^6}{2592} + \frac{10877 N^5}{2160} + \frac{8663 N^4}{1080} + \frac{14371 N^3}{1620} + \frac{143 N^2}{21} + \frac{1067 N}{315} + 1 \\
\text{if} \quad N \equiv 4 \pmod{6} \\
\\
\frac{8389 N^8}{120960} + \frac{8389 N^7}{15120} + \frac{5531 N^6}{2592} + \frac{10877 N^5}{2160} + \frac{69169 N^4}{8640} + \frac{57079 N^3}{6480} + \frac{4303 N^2}{672} + \frac{13607 N}{5040} + \frac{5045}{10368} \\
\text{if} \quad N \equiv 5 \pmod{6}
     \end{array}
   \right.$
}

\bigskip

\noindent which gives the exact number of points with integer coordinates in $K_4(N)$, i. e. the exact number of $4 \times 4$ $\mathbb{Z}$-magic squares with entries in $[0,N]$. The coefficient of the leading term, $8389/120960$, is the exact volume of $K_4(1)$ (one can confirm this value, for example, with \url{polymake}; see \url{http://www.polymake.org/}). You may wish to compare the quasipolynomials $E_3(N)$ and $E_4(N)$ with the quasipolynomials of period 2 computed in \cite{BCCG}, depending on the magic sum and not on a bound for the entries. Then, $\beta_{\infty, 4} (N) = \frac{8389 N^8}{120960}$.

\subsection{Volume for higher values of $n$} \label{volnxn}
It is reasonable to think that if for the case $3 \times 3$ we were able to interpolate our quasipolynomial by hand and for the case $4 \times 4$ we could do it with some computer time, then for the case $5 \times 5$ we would be able to interpolate the corresponding quasipolynomial with just a little more computer time. 

\bigskip

But this is not exactly the case. For the case of $5 \times 5$ $\mathbb{Z}$-magic squares, we would have a polytope, $K_5(1)$, in a 15-dimensional space defined by 25 pairs of inequalities given by $0 \le \psi_i(x) \le 1$ for $i = 1, \ldots, 25$. It is easy and fast to find some vertices of this polytope with denominators 3, 5, 7 and 8, for example
\begin{displaymath}
\begin{array}{ll}
(0,0,0,1,1,0,1,1,0,1,0,\frac{1}{3},0,1,\frac{1}{3}) & (0,0,0,1,1,1,0,1,0,0,\frac{4}{5},\frac{4}{5},0,1,\frac{3}{5})\\ (0,0,0,1,1,0,\frac{4}{7},1,\frac{3}{7},\frac{5}{7},1,\frac{2}{7},0,1,\frac{4}{7}) & (0,0,1,0,1,1,\frac{5}{8},0,\frac{3}{8},1,0,\frac{1}{4},\frac{3}{4},0,\frac{7}{8}).
\end{array}
\end{displaymath}
This means that the least common multiple of the denominators of the vertices of $K_5(1)$ is at least $3 \cdot 5 \cdot 7 \cdot 8 = 840$. To interpolate a quasipolynomial of degree 15 and period dividing 840 we would need $(15 + 1) 840 = 13440$ values. In order to calculate, for example, the largest one of these values, in principle, one would have to consider the points with integer coordinates in $[0,13339]^{15}$ and check if the 20 inequalities given by the nontrivial linear forms are satisfied. Even considering that we could reduce the number of values to roughly half (6719 values) with Ehrhart-Macdonald Reciprocity Law (see Appendix \ref{apenb}), and that with efficient programing the number of calculations can be further reduced, one would have to make a very very high number of checks. And of course this grows extremely quickly with $n$.

\bigskip

It is true that we could try finding the quasipolynomial with other ideas and also that for our problem we do not need the quasipolynomial $E_n(N)$ but only its leading coefficient, which gives the volume of $K_n(1)$. And it is also true that there are many other ways to calculate the volume of a polytope. But, to the best of our knowledge, no known method would give a formula for the exact volume of $K_n(1)$ for general $n$, and giving the exact value of the volume would be hard even for relatively small values of $n$. As an example illustrating this, for the Birkhoff polytope, one of the most important and studied polytopes, which is a similar but definitely simpler\footnote{The Birkhoff polytope $\mathcal{B}_n$ is the convex polytope in $\mathbb{R}^{n^2}$ whose points are the doubly stochastic matrices, i.e. the $n \times n$ matrices whose entries are non-negative real numbers and whose rows and columns each add up to 1. $\mathcal{B}_n$ is a convex polytope with integer vertices (which are well known).} polytope than the one we are considering, the volume is only known for $n \le 10$ and the case $n = 10$ took almost 17 years of computer time (see \cite{BP} and \url{http://www.math.binghamton.edu/dennis/Birkhoff/}).

\bigskip

So we cannot give a formula for the exact volume of the polytope $K_n(1)$, but we can prove that the volume will be a nonzero rational number. The fact that it will be a rational number is easy to see with our approach because the coefficients of the polynomials we are interpolating are given by the solutions of linear systems with integer coefficients.

\bigskip

To show that the volume is never zero, we first make the next observation:

\begin{lemma} \label{argumentosuma1}
Consider the linear forms given by the elephant basis for some $n \ge 5$. Then, for every $i = 1, \ldots, n^2$ the sum of the coefficients of the linear form $\psi_i$ is 1.
\end{lemma}

\begin{proof}
Since the sum of magic squares gives a magic square, when we add up the $n^2 - 2n$ vectors of our elephant basis we obtain a magic square. The obtained magic square has 1 as an entry in all the positions of the skeleton of the elephant basis. Then, this magic square is the only magic square with ones in all that positions, so it has to be the magic square with $n^2$ ones as entries. Since the coefficients of $\psi_i$ are the entries on the $i$-th position of the vectors of the elephant basis, we are done.
\end{proof}

\bigskip

Then, the point $(\frac{1}{2}, \ldots, \frac{1}{2}) \in \mathbb{R}^{n^2 - 2n}$ belongs to $K_n(N)$ because $\psi_i(\frac{1}{2}, \ldots, \frac{1}{2}) = 1/2$ for every $i = 1, \ldots, n^2$. Since all the coefficients of all our linear forms have absolute value less than or equal to $n$ (see, for example, Figure \ref{matelebasis}) then all the points $(\frac{1}{2}, \ldots, \frac{1}{2}) + \frac{1}{2n}e_j$ for $j = 1, \ldots n^2 - 2n$, where $e_j$ is the $j$-th vector of the canonical basis of the $(n^2 - 2n)$-dimensional Euclidean space, belong to $K_n(1)$. Because the vectors joining $(\frac{1}{2}, \ldots, \frac{1}{2})$ to these $n^2 - 2n$ points are linearly independent, we have that the convex hull of the $n^2 - 2n + 1$ points has nonzero volume, and because $K_n(N)$ is convex we have that it contains that convex hull. Then, $K_n(1)$ has nonzero volume.

\bigskip

Since the volume of $K_n(1)$ is less than or equal to one (in fact it is much smaller than one) because the inequalities given by the trivial linear forms define the unit cube, we have proved:

\begin{lemma}
For every $n \ge 5$, $\beta_{\infty, n} (N) = c(n) N^{n^2-2n}$ where $c(n) \in (0,1]$ is the rational number that gives the volume of $K_n(1)$. 
\end{lemma}

\section{The local factors}\label{sec5local}

The only ingredients we are missing in order to give our asymptotics are the local factors $\beta_{p,n}$. Recall that for each prime $p$,
$$\beta_{p,n} = \mathbb{E}_{m \in \mathbb{Z}_p^{n^2-2n}} \prod_{i=1}^{n^2} \Lambda_{\mathbb{Z}_p} (\psi_i(m)),$$
where $\Lambda_{\mathbb{Z}_p}$ is the \textit{local von Mangoldt function}, $\Lambda_{\mathbb{Z}_p}:\mathbb{Z}_p \rightarrow \mathbb{R}^+$, defined by $\Lambda_{\mathbb{Z}_p}(0) = 0$ and $\Lambda_{\mathbb{Z}_p}(b) = \frac{p}{\phi(p)} = \frac{p}{p-1}$ if $b \neq 0$.

\bigskip

Then, the important point is determining for which elements $m \in \mathbb{Z}_p^{n^2-2n}$, we have $\psi_i(m) \not \equiv 0 \pmod{p}$ for every $i = 1, 2, \ldots, n^2$. In order to do that we will use the inclusion-exclusion principle. As we will observe, for small primes we can have a different behaviour but one can always determine a nice formula giving the value of $\beta_{p}$ for every $p$ from some point on.

\subsection{Local factors for $n=3$}

We can think we have the system of 9 linear forms $\Psi: \mathbb{Z}_p^3 \to \mathbb{Z}_p^9$, defined by $\Psi(a,b,c)=(a+b,a-b-c,a+c,a-b+c,a,a+b-c,a-c,a+b+c,a-b)$, and we want to determine how many of the $p^3$ elements of $\mathbb{Z}_p^3$ give nonzero values modulo $p$ for all nine linear forms.

\bigskip

We use the inclusion-exclusion principle. There are $p^3$ elements in $\mathbb{Z}_p^3$. If none of the linear forms is trivial in $\mathbb{Z}_p^3$, any one of them is cancelled for $p^2$ values of $\mathbb{Z}_p^3$. We have to be careful because even if none of our linear forms is trivial in the integers, it could be the case that some of them were trivial when reducing modulo $p$ for some small values of $p$. Since all the coefficients of our linear forms are zeros, and plus or minus ones, this is not the case.

\bigskip

Now, if none of the linear forms is a multiple of another, any of the $\binom{9}{2}$ systems formed by two of the linear forms is cancelled for $p$ values of $\mathbb{Z}_p^3$. This is satisfied for every $p \ge 3$ but for $p = 2$ some of the linear forms are identical to others (for example, the second one is the same as the fourth, the sixth and the eighth). We will later study the case $p = 2$ separately.

\bigskip

The next step is computing, for every system of three of the linear forms, the rank of the matrix formed with their coefficients as rows. Again, small primes can be tricky. For example, if we select the last three linear forms, the corresponding matrix with integer coefficients and its Hermite normal form\footnote{There are different conventions to define the Hermite normal form of a matrix with coefficients in $\mathbb{Z}$. For example, we can allow the next three \textit{unimodular elementary row operations}: interchanging two rows, multiplying one row by $-1$ and adding an integer multiple of a row to another one. Then, the Hermite normal form of a matrix with coefficients in $\mathbb{Z}$ can be defined as the unique matrix obtained from it by \textit{unimodular elementary row operations} such that the first nonzero entry of each row is positive and strictly to the right of the first nonzero entry of the row on top of it and such that all the entries in the same column than the first nonzero entry of each row are less than it and greater than or equal to 0. } are respectively:
\[ \left( \begin{array}{ccc}
1 & 0 & -1 \\
1 & 1 & 1 \\
1 & -1 & 0 \end{array} \right)
\text{ and }
\left( \begin{array}{ccc}
1 & 0 & 2 \\
0 & 1 & 2 \\
0 & 0 & 3 \end{array} \right).\] 

Then, the matrix in our example has rank 2 in $\mathbb{Z}_3$ and has rank 3 in $\mathbb{Z}_p$ for every $p > 3$. We will also later study the case $p = 3$ separately. It makes sense to compute the Hermite normal form of all the matrices with the coefficients of 3, 4, 5, 6, 7, 8 or the 9 forms as rows. If we then look at the largest of all the first nonzero elements of a row we will see that it is equal to 4. Then, we know that for every prime $p \ge 5$ the Hermite normal form of these matrices in $\mathbb{Z}$ has exactly the same rank when we reduce its coefficients modulo $p$. This means that we can study at the same time $\beta_{p,3}$ for every prime $p \ge 5$.

\bigskip

We have that 8 of the matrices with three of the linear forms as rows have rank 2 and the other 76 have rank 3 in $\mathbb{Z}$. Also, any matrix with 4, 5, 6, 7, 8 or the 9 forms as rows has rank 3 in $\mathbb{Z}$. This means that we can use the inclusion-exclusion principle to deduce that for every prime $p \ge 5$ the number of elements $m \in \mathbb{Z}_p^3$, with $\psi_i(m) \not \equiv 0 \pmod{p}$ for every $i = 1, 2, \ldots, 9$ is:
$$p^3 - 9p^2 + 36p - (8p + 76) + 126 - 126 + 84 - 36 + 9 -1 = p^3 - 9p^2 + 28p - 20.$$

Then, for every prime $p \ge 5$
$$\beta_{p,3} = \frac{p^3 - 9p^2 + 28p - 20}{p^3} \left( \frac{p}{p-1} \right)^9.$$

\bigskip

For $p = 2$, the four elements of $\mathbb{Z}_2^3$ starting with 0 cancel the fifth linear form. The element $(1,0,0)$ does not cancel any of the nine linear forms. The other three elements starting with one either cancel the first or the third linear form. Then, $\beta_{2,3}= \frac{1}{2^3} \cdot 2^9 = 2^6$. Similarly, for $p=3$ it is easy to see that the only two elements of $\mathbb{Z}_3^3$ which do not cancel any of the nine linear forms are $(1,0,0)$ and $(2,0,0)$. Then $\beta_{3,3}= \frac{2}{3^3} \cdot \left( \frac{3}{2} \right)^9 = \frac{3^6}{2^8}$.

\subsection{Local factors for $n=4$}

In this case we have the 16 linear forms which coefficients are given by the rows of the matrix of Figure \ref{matele4}. We will determine how many of the $p^8$ elements of $\mathbb{Z}_p^8$ give nonzero values modulo $p$ for all the sixteen linear forms.

\bigskip

First of all we compute the Hermite normal form of all the matrices formed with some of the rows of the matrix of Figure \ref{matele4}. The largest of all the first nonzero elements of a row appearing in all these matrices is $3$, so, as before, for every prime $p \ge 5$ the Hermite normal form of these matrices in $\mathbb{Z}$ has exactly the same rank when we reduce its coefficients modulo $p$. 

\bigskip

Now we compute the ranks in $\mathbb{Z}$ of all these matrices and we have: all 16 one row matrices have rank 1, all 120 two row matrices have rank 2 and all 560 three row matrices have rank 3; 10 four row matrices have rank 3 and 1810 have rank 4; 120 five row matrices have rank 4 and 4248 have rank 5; 708 six row matrices have rank 5 and 7300 have rank 6; 32 seven row matrices have rank 5, 2656 have rank 6 and 8752 have rank 7; 433 eight row matrices have rank 6, 6553 have rank 7 and 5884 have rank 8; 32 nine row matrices have rank 6, 2656 have rank 7 and 8752 have rank 8; 708 ten row matrices have rank 7 and 7300 have rank 8; 120 eleven row matrices have rank 7 and 4248 have rank 8; 10 twelve row matrices have rank 7 and 1810 have rank 8; all 560 thirteen row matrices have rank 8 and the same thing happens with the 120 fourteen row matrices, the 16 fifteen row matrices and the only sixteen row matrix. By the inclusion-exclusion principle, for every $p \ge 5$ the number of elements $m \in \mathbb{Z}_p^8$, with $\psi_i(m) \not \equiv 0 \pmod{p}$ for every $i = 1, 2, \ldots, 16$ is:
$$p^8 - 16p^7 + 120p^6 - 550p^5 + 1690p^4 - 3572p^3 + 5045p^2 - 4257p + 1539.$$

Then, for every prime $p \ge 5$
\begin{small}
$$\beta_{p,4} = \frac{p^8 - 16p^7 + 120p^6 - 550p^5 + 1690p^4 - 3572p^3 + 5045p^2 - 4257p +1539}{p^8} \left( \frac{p}{p-1} \right)^{16}.$$
\end{small}

There is $1$ element of $\mathbb{Z}_2^8$ and there are $34$ elements of $\mathbb{Z}_3^8$ not cancelling any of the sixteen linear forms, so
$$\beta_{2,4} = \frac{1}{2^8} \cdot 2^{16} = 2^8 \qquad \text{and} \qquad \beta_{3,4} = \frac{34}{3^8} \cdot \left( \frac{3}{2} \right)^{16} = \frac{17 \cdot 3^8}{2^{15}}.$$

\subsection{Local factors for higher values of $n$}
In the general case, it is clear that we can follow the same steps. Since every two of our linear forms are linearly independent, for primes $p \ge p_0(n)$ we will have:
\begin{displaymath}
\begin{array}{lll}
\beta_{p,n} & = & \dfrac{p^{n^2-2n} - n^2p^{n^2-2n-1} + P_{\le n^2-2n-2}(p)}{p^{n^2-2n}} \left( \dfrac{p}{p-1} \right)^{n^2} \\
& = & \dfrac{p^{2n^2-2n} - n^2p^{2n^2-2n-1} + P_{\le 2n^2-2n-2}(p)}{p^{2n^2-2n} - n^2p^{2n^2-2n-1} + P_{=2n^2-2n-2}(p)} \\
& = & 1 + \dfrac{P_{\le 2n^2-2n-2}(p)}{P_{=2n^2-2n}(p)},
\end{array}
\end{displaymath}
where $P_{\le D}(p)$ and $P_{=D}(p)$ are (possibly different in each occasion) polynomials in $p$ of degree less than or equal to $D$ and $D$ respectively.

\bigskip

The existence of the logarithm function makes the theory of infinite products essentially equivalent to the theory of infinite series. In particular, it is well known that (see, for example, \cite{KK}):
\begin{lemma}
If $\{a_i\}_{i=1}^{\infty}$ is a sequence of real numbers such that $0 \le |a_i| < 1$ and $\sum_{i=1}^{\infty}|a_i| < \infty$ then $\prod_{i=1}^{\infty} (1 + a_i)$ converges to a nonzero real number.
\end{lemma}

We know that for every prime, $p$, we have $\beta_{p,n} > 0$ because $m = (1, \ldots, 1) \in \mathbb{Z}^{n^2 - 2n}$ makes $\psi_i(m) = 1$ for every $i = 1, 2, \ldots, n^2$. Also, since there is a constant, $C(n) > 0$, such that for $p \ge p_0(n)$ we have
$$\left| \dfrac{P_{\le 2n^2-2n-2}(p)}{P_{=2n^2-2n}(p)} \right| \le \frac{C(n)}{n^2},$$
then, by the previous lemma, we deduce:
\begin{lemma}
For every $n \ge 3$, $\prod_{p \text{ prime}} \beta_{p,n}$ converges to a strictly positive real number.
\end{lemma}

\section{Asymptotics for Magic Squares of Primes}\label{sec6asym}
The work of the previous sections and the results mentioned in the Introduction prove the next three theorems:

\begin{theorem}
The number of $3 \times 3$ $\mathbb{Z}$-magic squares with their entries being prime numbers in $[0,N]$ is
$$(1 + o(1))\mathfrak{S}_3 \frac{N^3}{\log^9N},$$
where
$$\mathfrak{S}_3 = \frac{243}{8} \prod_{\substack{p \text{ prime} \\ p \ge 5}} \frac{p^3 - 9p^2 + 28p - 20}{p^3} \left( \frac{p}{p-1} \right)^9 \approx 25.818.$$
\end{theorem}

Observe that, apart from the asymptotics for 5-term arithmetic progressions of primes, this is one of the first ``natural'' applications of the work of Green, Tao and Ziegler to a system of linear equations of complexity 3 (see Section \ref{seccomp}).

\bigskip

For $n = 4$ we have complexity 1 (see Section \ref{seccomp}) and the asymptotic:

\begin{theorem}
The number of $4 \times 4$ $\mathbb{Z}$-magic squares with their entries being prime numbers in $[0,N]$ is
$$(1 + o(1))\mathfrak{S}_4 \frac{N^8}{\log^{16}N},$$
where
\begin{small}
\begin{displaymath}
\begin{array}{lll}
\mathfrak{S}_4 & = & \frac{34654959}{573440} \prod_{\substack{p \text{ prime} \\ p \ge 5}} \frac{p^8 - 16p^7 + 120p^6 - 550p^5 + 1690p^4 - 3572p^3 + 5045p^2 - 4257p +1539}{p^8} \left( \frac{p}{p-1} \right)^{16} \\
& \approx & 76.758.
\end{array}
\end{displaymath}
\end{small}
\end{theorem}

Finally, for $n \ge 5$ (see Section \ref{seccomp}) we have complexity 1 and the ``order of magnitude of the asymptotics'':

\begin{theorem}
The number of $n \times n$ $\mathbb{Z}$-magic squares with their entries being prime numbers in $[0,N]$ is
$$(1 + o(1))\mathfrak{S}_n \frac{N^{n^2-2n}}{\log^{n^2}N},$$
where $\mathfrak{S}_n = \text{vol}_{n^2-2n}(K_n(1)) \prod_{p \text{ prime}} \beta_{p,n}$ is a nonzero real number.
\end{theorem}

\section{Magic squares with different entries}\label{sec7dif}

Up to this point we have not been careful with the fact that there could be repetitions among the entries of our $\mathbb{Z}$-magic squares. In this section we prove that the asymptotics for $\mathbb{Z}$-magic squares of primes with different entries are exactly the same as the asymptotics for $\mathbb{Z}$-magic squares of primes. We first observe:

\begin{lemma} \label{cuadmagdif}
For every $n \ge 3$, there are well known constructions of $n \times n$ $\mathbb{Z}$-magic squares with their $n^2$ entries being the first $n^2$ natural numbers (these are usually called normal magic squares).
\end{lemma}

Maybe the best known constructions for normal magic squares of odd side are the ones popularly known today by the names of Simon de la Loub\`{e}re and Bachet de M\`{e}ziriac. For normal magic squares of even side there are also different constructions, for example those of Devedec to name some. All these constructions can be found in \cite{K} and many others very easily searching on the Internet.

\bigskip

The only important thing for us is that there exist magic squares with all their entries being different. With this in mind we can prove the next lemma.

\begin{lemma}
For every $n \ge 3$ the number of $n \times n$ $\mathbb{Z}$-magic squares with their entries being primes in $[0,N]$ is asymptotically equal to the number of $n \times n$ $\mathbb{Z}$-magic squares with their entries being different primes in $[0,N]$ when $N \to \infty$.
\end{lemma}

\begin{proof}
In the same way that we proved that the $2n$ equations defining $n \times n$ magic squares were linearly independent exhibiting magic squares that satisfied some of the equations but not the others (see page \pageref{elegante}), the fact that there are $\mathbb{Z}$-magic squares with all their entries different (Lemma \ref{cuadmagdif}) proves that any of the $\binom{n^2}{2}$ equations of the type $x_i - x_{j} = 0$ for $i \in [1,n^2-1]$ and $j \in [i+1, n^2]$ is linearly independent with the $2n$ equations defining $n \times n$ magic squares (see page \pageref{sistlineqA}).

\bigskip

Then, the number of $\mathbb{Z}$-magic squares with their entries being different elements of $[0,N]$ is $O_n(N^{n^2-2n-1})$. This is certainly $o_n(N^{n^2-2n}/{\log ^{n^2} N})$ and then we are done.
\end{proof}

\section*{Acknowledgements}
The author would like to thank Fundaci\'on Ram\'on Areces for the Postdoctoral Grant he enjoys at the Department of Pure Mathematics and Mathematical Statistics of the University of Cambridge. He would also like to thank Javier Cilleruelo for the proposal of the problem, Ben Green for some valuable suggestions and both of them for their encouragement.

\appendix
\section{Vertices of $K_4(1)$} \label{apena}
The 178 vertices of $K_4(1)$ in lexicographic order are:
\begin{footnotesize}
\begin{displaymath}
\begin{array}{llll}
(0, 0, 0, 0, 0, 0, 0, 0) & (0, 0, 0, 1, 0, 1, 0, 1) & (0, 0, 0, 1, 1, 0, 0, 0) & (0, 0, \frac{1}{2}, \frac{1}{2}, 0, \frac{1}{2}, \frac{1}{2}, 1)\\ 
(0, 0, \frac{1}{2}, 1, \frac{1}{2}, \frac{1}{2}, \frac{1}{2}, 1) & (0, 0, 1, 0, 0, \frac{1}{2}, 0, \frac{1}{2}) & (0, 0, 1, 0, 0, 1, 0, 0) & (0, 0, 1, 0, 1, 0, 0, 0)\\ 
(0, 0, 1, \frac{1}{3}, 0, \frac{2}{3}, \frac{1}{3}, 1) & (0, 0, 1, \frac{1}{2}, 0, 1, 0, 1) & (0, 0, 1, \frac{1}{2}, 0, 1, \frac{1}{2}, 1) & (0, 0, 1, \frac{1}{2}, \frac{1}{2}, \frac{1}{2}, \frac{1}{2}, 1)\\ 
(0, 0, 1, 1, 0, 1, 0, 1) & (0, 0, 1, 1, \frac{1}{2}, 1, \frac{1}{2}, 1) & (0, 0, 1, 1, 1, \frac{1}{2}, 0, \frac{1}{2}) & (0, 0, 1, 1, 1, \frac{1}{2}, \frac{1}{2}, 1)\\ 
(0, 0, 1, 1, 1, 1, 0, 0) & (0, 0, 1, 1, 1, 1, 0, 1) & (0, \frac{1}{2}, 0, \frac{1}{2}, 1, 0, 0, 0) & (0, \frac{1}{2}, 0, 1, 1, \frac{1}{2}, 0, \frac{1}{2})\\ 
(0, \frac{1}{2}, 1, 0, 0, \frac{1}{2}, 0, \frac{1}{2}) & (0, \frac{1}{2}, 1, 0, 0, \frac{1}{2}, \frac{1}{2}, 1) & (0, \frac{1}{2}, 1, \frac{1}{2}, 0, 1, 0, 1) & (0, \frac{1}{2}, 1, \frac{1}{2}, 1, 1, 0, 0)\\ 
(0, \frac{1}{2}, 1, 1, 1, \frac{1}{2}, \frac{1}{2}, 1) & (0, 1, 0, 0, 0, 0, 0, 0) & (0, 1, 0, 0, 0, 0, 1, 1) & (0, 1, 0, 0, \frac{1}{2}, 0, \frac{1}{2}, 0)\\ 
(0, 1, 0, \frac{1}{3}, 1, 0, \frac{1}{3}, 0) & (0, 1, 0, \frac{1}{2}, 1, 0, 0, 0) & (0, 1, 0, \frac{1}{2}, 1, 0, \frac{1}{2}, 0) & (0, 1, 0, \frac{1}{2}, 1, 0, \frac{1}{2}, \frac{1}{2})\\ 
(0, 1, 0, 1, 0, 1, 0, 1) & (0, 1, 0, 1, \frac{1}{2}, 1, \frac{1}{2}, 1) & (0, 1, 0, 1, 1, 0, 1, 1) & (0, 1, 0, 1, 1, \frac{1}{2}, 0, \frac{1}{2})\\ 
(0, 1, 0, 1, 1, \frac{1}{2}, \frac{1}{2}, 1) & (0, 1, 0, 1, 1, 1, 0, 0) & (0, 1, \frac{1}{2}, 0, \frac{1}{2}, \frac{1}{2}, \frac{1}{2}, 0) & (0, 1, \frac{1}{2}, 0, 1, 0, \frac{1}{2}, 0)\\ 
(0, 1, \frac{1}{2}, \frac{1}{2}, 0, \frac{1}{2}, \frac{1}{2}, 1) & (0, 1, \frac{1}{2}, \frac{1}{2}, 1, \frac{1}{2}, \frac{1}{2}, 0) & (0, 1, \frac{1}{2}, 1, 1, 1, \frac{1}{2}, 1) & (0, 1, 1, 0, 0, \frac{1}{2}, \frac{1}{2}, 1)\\ 
(0, 1, 1, 0, 0, 1, 1, 1) & (0, 1, 1, 0, 1, 0, 0, 0) & (0, 1, 1, 0, 1, \frac{1}{2}, \frac{1}{2}, 0) & (0, 1, 1, \frac{1}{2}, \frac{1}{2}, \frac{1}{2}, \frac{1}{2}, 1)\\ 
(0, 1, 1, \frac{1}{2}, 1, 1, \frac{1}{2}, \frac{1}{2}) & (0, 1, 1, 1, 1, 1, 0, 1) & (0, 1, 1, 1, 1, 1, 1, 1) & (\frac{1}{3}, 0, \frac{1}{3}, \frac{2}{3}, 0, \frac{1}{3}, \frac{2}{3}, 1)\\ 
(\frac{1}{3}, 0, 1, 0, 0, \frac{1}{3}, 0, \frac{1}{3}) & (\frac{1}{3}, \frac{1}{3}, 0, \frac{2}{3}, 1, \frac{1}{3}, 0, 0) & (\frac{1}{3}, \frac{2}{3}, 1, \frac{2}{3}, \frac{2}{3}, \frac{1}{3}, \frac{2}{3}, 1) & (\frac{1}{3}, 1, 0, 0, \frac{1}{3}, \frac{1}{3}, \frac{2}{3}, 0)\\ 
(\frac{1}{3}, 1, \frac{2}{3}, \frac{2}{3}, 1, 1, \frac{2}{3}, \frac{2}{3}) & (\frac{1}{2}, 0, 0, \frac{1}{2}, 0, 0, \frac{1}{2}, \frac{1}{2}) & (\frac{1}{2}, 0, 0, \frac{1}{2}, \frac{1}{2}, \frac{1}{2}, 0, 0) & (\frac{1}{2}, 0, 0, 1, 0, \frac{1}{2}, \frac{1}{2}, 1)\\ 
(\frac{1}{2}, 0, 0, 1, 1, \frac{1}{2}, 0, 0) & (\frac{1}{2}, 0, \frac{1}{2}, 0, 0, 0, 0, 0) & (\frac{1}{2}, 0, \frac{1}{2}, \frac{1}{2}, 0, \frac{1}{2}, 1, 1) & (\frac{1}{2}, 0, \frac{1}{2}, \frac{1}{2}, \frac{1}{2}, 0, 0, 0)\\ 
(\frac{1}{2}, 0, \frac{1}{2}, \frac{1}{2}, \frac{1}{2}, 1, 0, 0) & (\frac{1}{2}, 0, \frac{1}{2}, \frac{1}{2}, 1, \frac{1}{2}, 0, 0) & (\frac{1}{2}, 0, \frac{1}{2}, 1, 0, 1, 0, 1) & (\frac{1}{2}, 0, \frac{1}{2}, 1, 1, 1, 0, 0)\\ 
(\frac{1}{2}, 0, 1, 0, 0, \frac{1}{2}, 0, 0) & (\frac{1}{2}, 0, 1, 0, 0, \frac{1}{2}, 0, \frac{1}{2}) & (\frac{1}{2}, 0, 1, 0, \frac{1}{2}, 0, 0, 0) & (\frac{1}{2}, 0, 1, \frac{1}{2}, 0, 1, 0, 1)\\ 
(\frac{1}{2}, 0, 1, \frac{1}{2}, 0, 1, \frac{1}{2}, \frac{1}{2}) & (\frac{1}{2}, 0, 1, \frac{1}{2}, 0, 1, 1, 1) & (\frac{1}{2}, 0, 1, \frac{1}{2}, \frac{1}{2}, \frac{1}{2}, 0, 0) & (\frac{1}{2}, 0, 1, \frac{1}{2}, \frac{1}{2}, \frac{1}{2}, 1, 1)\\ 
(\frac{1}{2}, 0, 1, \frac{1}{2}, 1, 0, 0, 0) & (\frac{1}{2}, 0, 1, \frac{1}{2}, 1, 0, \frac{1}{2}, \frac{1}{2}) & (\frac{1}{2}, 0, 1, \frac{1}{2}, 1, 1, 0, 0) & (\frac{1}{2}, 0, 1, 1, \frac{1}{2}, 1, 0, 1)\\ 
(\frac{1}{2}, 0, 1, 1, 1, \frac{1}{2}, 0, \frac{1}{2}) & (\frac{1}{2}, 0, 1, 1, 1, \frac{1}{2}, \frac{1}{2}, 1) & (\frac{1}{2}, \frac{1}{2}, 0, 0, 0, \frac{1}{2}, \frac{1}{2}, 0) & (\frac{1}{2}, \frac{1}{2}, 0, \frac{1}{2}, 0, 0, \frac{1}{2}, \frac{1}{2})\\ 
(\frac{1}{2}, \frac{1}{2}, 0, \frac{1}{2}, 0, 0, 1, 1) & (\frac{1}{2}, \frac{1}{2}, 0, \frac{1}{2}, 0, 1, \frac{1}{2}, \frac{1}{2}) & (\frac{1}{2}, \frac{1}{2}, 0, \frac{1}{2}, 1, 0, 0, 0) & (\frac{1}{2}, \frac{1}{2}, 0, 1, 0, \frac{1}{2}, \frac{1}{2}, 1)\\ 
(\frac{1}{2}, \frac{1}{2}, 0, 1, 1, \frac{1}{2}, \frac{1}{2}, 0) & (\frac{1}{2}, \frac{1}{2}, 1, 0, 0, \frac{1}{2}, \frac{1}{2}, 1) & (\frac{1}{2}, \frac{1}{2}, 1, 0, 1, \frac{1}{2}, \frac{1}{2}, 0) & (\frac{1}{2}, \frac{1}{2}, 1, \frac{1}{2}, 0, 1, 1, 1)\\ 
(\frac{1}{2}, \frac{1}{2}, 1, \frac{1}{2}, 1, 0, \frac{1}{2}, \frac{1}{2}) & (\frac{1}{2}, \frac{1}{2}, 1, \frac{1}{2}, 1, 1, 0, 0) & (\frac{1}{2}, \frac{1}{2}, 1, \frac{1}{2}, 1, 1, \frac{1}{2}, \frac{1}{2}) & (\frac{1}{2}, \frac{1}{2}, 1, 1, 1, \frac{1}{2}, \frac{1}{2}, 1)\\ 
(\frac{1}{2}, 1, 0, 0, 0, \frac{1}{2}, \frac{1}{2}, 0) & (\frac{1}{2}, 1, 0, 0, 0, \frac{1}{2}, 1, \frac{1}{2}) & (\frac{1}{2}, 1, 0, 0, \frac{1}{2}, 0, 1, 0) & (\frac{1}{2}, 1, 0, \frac{1}{2}, 0, 0, 1, 1)\\ 
(\frac{1}{2}, 1, 0, \frac{1}{2}, 0, 1, \frac{1}{2}, \frac{1}{2}) & (\frac{1}{2}, 1, 0, \frac{1}{2}, 0, 1, 1, 1) & (\frac{1}{2}, 1, 0, \frac{1}{2}, \frac{1}{2}, \frac{1}{2}, 0, 0) & (\frac{1}{2}, 1, 0, \frac{1}{2}, \frac{1}{2}, \frac{1}{2}, 1, 1)\\ 
(\frac{1}{2}, 1, 0, \frac{1}{2}, 1, 0, 0, 0) & (\frac{1}{2}, 1, 0, \frac{1}{2}, 1, 0, \frac{1}{2}, \frac{1}{2}) & (\frac{1}{2}, 1, 0, \frac{1}{2}, 1, 0, 1, 0) & (\frac{1}{2}, 1, 0, 1, \frac{1}{2}, 1, 1, 1)\\ 
(\frac{1}{2}, 1, 0, 1, 1, \frac{1}{2}, 1, \frac{1}{2}) & (\frac{1}{2}, 1, 0, 1, 1, \frac{1}{2}, 1, 1) & (\frac{1}{2}, 1, \frac{1}{2}, 0, 0, 0, 1, 1) & (\frac{1}{2}, 1, \frac{1}{2}, 0, 1, 0, 1, 0)\\ 
(\frac{1}{2}, 1, \frac{1}{2}, \frac{1}{2}, 0, \frac{1}{2}, 1, 1) & (\frac{1}{2}, 1, \frac{1}{2}, \frac{1}{2}, \frac{1}{2}, 0, 1, 1) & (\frac{1}{2}, 1, \frac{1}{2}, \frac{1}{2}, \frac{1}{2}, 1, 1, 1) & (\frac{1}{2}, 1, \frac{1}{2}, \frac{1}{2}, 1, \frac{1}{2}, 0, 0)\\ 
(\frac{1}{2}, 1, \frac{1}{2}, 1, 1, 1, 1, 1) & (\frac{1}{2}, 1, 1, 0, 0, \frac{1}{2}, 1, 1) & (\frac{1}{2}, 1, 1, 0, 1, \frac{1}{2}, \frac{1}{2}, 0) & (\frac{1}{2}, 1, 1, \frac{1}{2}, \frac{1}{2}, \frac{1}{2}, 1, 1)\\ 
(\frac{1}{2}, 1, 1, \frac{1}{2}, 1, 1, \frac{1}{2}, \frac{1}{2}) & (\frac{2}{3}, 0, \frac{1}{3}, \frac{1}{3}, 0, 0, \frac{1}{3}, \frac{1}{3}) & (\frac{2}{3}, 0, 1, 1, \frac{2}{3}, \frac{2}{3}, \frac{1}{3}, 1) & (\frac{2}{3}, \frac{1}{3}, 0, \frac{1}{3}, \frac{1}{3}, \frac{2}{3}, \frac{1}{3}, 0)\\ 
(\frac{2}{3}, \frac{2}{3}, 1, \frac{1}{3}, 0, \frac{2}{3}, 1, 1) & (\frac{2}{3}, 1, 0, 1, 1, \frac{2}{3}, 1, \frac{2}{3}) & (\frac{2}{3}, 1, \frac{2}{3}, \frac{1}{3}, 1, \frac{2}{3}, \frac{1}{3}, 0) & (1, 0, 0, 0, 0, 0, 0, 0)\\ 
(1, 0, 0, 0, 0, 0, 1, 0) & (1, 0, 0, \frac{1}{2}, 0, 0, \frac{1}{2}, \frac{1}{2}) & (1, 0, 0, \frac{1}{2}, \frac{1}{2}, \frac{1}{2}, \frac{1}{2}, 0) & (1, 0, 0, 1, 0, \frac{1}{2}, \frac{1}{2}, 1)\\ 
(1, 0, 0, 1, 0, 1, 1, 1) & (1, 0, 0, 1, 1, 0, 0, 0) & (1, 0, 0, 1, 1, \frac{1}{2}, \frac{1}{2}, 0) & (1, 0, \frac{1}{2}, 0, 0, 0, \frac{1}{2}, 0)\\
(1, 0, \frac{1}{2}, \frac{1}{2}, 0, \frac{1}{2}, \frac{1}{2}, 1) & (1, 0, \frac{1}{2}, \frac{1}{2}, 1, \frac{1}{2}, \frac{1}{2}, 0) & (1, 0, \frac{1}{2}, 1, 0, 1, \frac{1}{2}, 1) & (1, 0, \frac{1}{2}, 1, \frac{1}{2}, \frac{1}{2}, \frac{1}{2}, 1)\\
(1, 0, 1, 0, 0, 0, 1, 1) & (1, 0, 1, 0, 0, \frac{1}{2}, \frac{1}{2}, 0) & (1, 0, 1, 0, 0, \frac{1}{2}, 1, \frac{1}{2}) & (1, 0, 1, 0, 0, 1, 0, 0)\\ 
(1, 0, 1, 0, \frac{1}{2}, 0, \frac{1}{2}, 0) & (1, 0, 1, 0, 1, 0, 1, 0) & (1, 0, 1, \frac{1}{2}, 0, 1, \frac{1}{2}, \frac{1}{2}) & (1, 0, 1, \frac{1}{2}, 0, 1, \frac{1}{2}, 1)\\ 
(1, 0, 1, \frac{1}{2}, 0, 1, 1, 1) & (1, 0, 1, \frac{2}{3}, 0, 1, \frac{2}{3}, 1) & (1, 0, 1, 1, \frac{1}{2}, 1, \frac{1}{2}, 1) & (1, 0, 1, 1, 1, 1, 0, 0)\\ 
(1, 0, 1, 1, 1, 1, 1, 1) & (1, \frac{1}{2}, 0, 0, 0, \frac{1}{2}, \frac{1}{2}, 0) & (1, \frac{1}{2}, 0, \frac{1}{2}, 0, 0, 1, 1) & (1, \frac{1}{2}, 0, \frac{1}{2}, 1, 0, 1, 0)
\end{array}
\end{displaymath}

\begin{displaymath}
\begin{array}{llll}
(1, \frac{1}{2}, 0, 1, 1, \frac{1}{2}, \frac{1}{2}, 0) & (1, \frac{1}{2}, 0, 1, 1, \frac{1}{2}, 1, \frac{1}{2}) & (1, \frac{1}{2}, 1, 0, 0, \frac{1}{2}, 1, \frac{1}{2}) & (1, \frac{1}{2}, 1, \frac{1}{2}, 0, 1, 1, 1)\\ 
(1, 1, 0, 0, 0, 0, 1, 0) & (1, 1, 0, 0, 0, 0, 1, 1) & (1, 1, 0, 0, 0, \frac{1}{2}, \frac{1}{2}, 0) & (1, 1, 0, 0, 0, \frac{1}{2}, 1, \frac{1}{2})\\ 
(1, 1, 0, 0, \frac{1}{2}, 0, \frac{1}{2}, 0) & (1, 1, 0, 0, 1, 0, 1, 0) & (1, 1, 0, \frac{1}{2}, \frac{1}{2}, \frac{1}{2}, \frac{1}{2}, 0) & (1, 1, 0, \frac{1}{2}, 1, 0, \frac{1}{2}, 0)\\ 
(1, 1, 0, \frac{1}{2}, 1, 0, 1, 0) & (1, 1, 0, \frac{2}{3}, 1, \frac{1}{3}, \frac{2}{3}, 0) & (1, 1, 0, 1, 0, 1, 1, 1) & (1, 1, 0, 1, 1, 0, 1, 1)\\ 
(1, 1, 0, 1, 1, \frac{1}{2}, 1, \frac{1}{2}) & (1, 1, \frac{1}{2}, 0, \frac{1}{2}, \frac{1}{2}, \frac{1}{2}, 0) & (1, 1, \frac{1}{2}, \frac{1}{2}, 1, \frac{1}{2}, \frac{1}{2}, 0) & (1, 1, 1, 0, 0, 1, 1, 1)\\ 
(1, 1, 1, 0, 1, 0, 1, 0) & (1, 1, 1, 1, 1, 1, 1, 1) & &
\end{array}
\end{displaymath}
\end{footnotesize}

These vertices were obtained with the algorithm described in Section \ref{vol4x4}. Exactly the same vertices are obtained using \url{cddlib}, the implementation of the double description method of Motzkin, Raiffa, Thompson and Thrall (Copyright by Komei Fukuda, \url{http://www.ifor.math.ethz.ch/~fukuda/cdd_home/cdd.html}).

\section{54 values of $E_4(N)$} \label{apenb}
Recall that $E_4(N)$ is a quasipolynomial giving the exact number of integer points that satisfy $0 \le \psi_i(x) \le N$ for $i = 1, \ldots, 16$, where the coefficients of $\psi_i(x)$ are given by the $i$-th row of the matrix of Figure \ref{matele4}. Since the degree of $E_4(N)$ is 8 and its period divides 6, we need 54 values of $E_4(N)$ in order to interpolate it.

\bigskip

Computing each one of these values essentially consists of checking for each one of the integer points in $[0,N]^8$ if it satisfies the 16 inequalities given by the eight nontrivial linear forms being between 0 and $N$, including both of them. It is true that we can be clever and reduce considerably the number of checkings, but this gives us an idea of the order of magnitude of the number of checkings needed. With this in mind, it would be nice if we could reduce the number of values of $E_4(N)$ needed to interpolate it. The next fundamental result (see, for example, \cite{BS} or \cite{MM}) allow us to reduce significantly that number of values, from 54 to roughly the first half of them.

\begin{theorem} \label{ehrmac}
\emph{\textbf{(Ehrhart-Macdonald Reciprocity Law)}} Suppose that $P$ is a convex polytope whose vertices have rational coordinates. Then the evaluation of its Ehrhart quasipolynomial, $E(N)$, at negative integers yields
$$E(-N) = (-1)^{\text{dim} P} E^{\circ}(N),$$
where $E^{\circ}(N)$ is the number of integer points in the interior of the dilation $NP$.
\end{theorem}

How to obtain the values of $E^{\circ}$ from the values of $E$? $E^{\circ}(N)$ is the number of points with integer coordinates in the interior of the polytope $K_n(N) = \{x \in \mathbb{R}^{n^2-2n} : 0 \le \psi_i(x) \le N \text{ for } i = 1, \ldots, n^2\}$. A point is in the interior of $K_n(N)$ if and only if it satisfies those $2n^2$ inequalities but none of the corresponding equalities, so $K_n^{\circ}(N) = \{x \in \mathbb{R}^{n^2-2n} : 0 < \psi_i(x) < N \text{ for } i = 1, \ldots, n^2\} = \{x \in \mathbb{R}^{n^2-2n} : 1 \le \psi_i(x) \le N-1 \text{ for } i = 1, \ldots, n^2\}$. 

\bigskip
Because of the property that for every $i = 1, \ldots, n^2$ the coefficients of $\psi_i$ add up to 1 (see Lemma \ref{argumentosuma1}), substracting 1 to all of the coordinates of the elements of $K_n^{\circ}(N)$ gives a biyection between this set and $\{x \in \mathbb{R}^{n^2-2n} : 0 \le \psi_i(x) \le N-2 \text{ for } i = 1, \ldots, n^2\} = K_{n}(N-2)$ for $N \ge 2$. Then, Theorem \ref{ehrmac} implies that in our case
$$E_4(-N) = E_4(N-2) \text{ for } N \ge 2.$$

\bigskip

This, and a little program checking the number of points $x \in [0,N]^8$ that satisfy $0 \le \psi_i(x) \le N$ for the $2n$ nontrivial linear forms and $N \in [0,26]$, together with the fact that $E^{\circ}(1) = 0$, gives:

\begin{displaymath}
\begin{array}{ll}
E_4(-1) = 0 & \quad E_4(0) = 1 = E_4(-2) \\
E_4(1) = 34 = E_4(-3) & \quad E_4(2) = 621 = E_4(-4) \\
E_4(3) = 5400 = E_4(-5) & \quad E_4(4) = 30277 = E_4(-6) \\
E_4(5) = 125794 = E_4(-7) & \quad E_4(6) = 423097 = E_4(-8) \\
E_4(7) = 1214992 = E_4(-9) & \quad E_4(8) = 3089369 = E_4(-10) \\
E_4(9) = 7130034 = E_4(-11) & \quad E_4(10) = 15210869 = E_4(-12) \\
E_4(11) = 30399592 = E_4(-13) & \quad E_4(12) = 57508653 = E_4(-14) \\
E_4(13) = 103807042 = E_4(-15) & \quad E_4(14) = 179946753 = E_4(-16) \\
E_4(15) = 301109616 = E_4(-17) & \quad E_4(16) = 488451089 = E_4(-18) \\
E_4(17) = 770830866 = E_4(-19) & \quad E_4(18) = 1186938765 = E_4(-20) \\
E_4(19) = 1787779544 = E_4(-21) & \quad E_4(20) = 2639668773 = E_4(-22) \\
E_4(21) = 3827663858 = E_4(-23) & \quad E_4(22) = 5459641001 = E_4(-24) \\
E_4(23) = 7670885920 = E_4(-25) & \quad E_4(24) = 10629486297 = E_4(-26) \\
E_4(25) = 14542317074 = E_4(-27) & \quad E_4(26) = 19662006197
\end{array}
\end{displaymath}

Although in this case we can directly calculate with a computer the 54 values $E_4(0), \ldots, E_{53}(0)$ --and we have done so!--, this takes around 217 times the time it takes to calculate the above values. This explicitly shows the ``power'' of Ehrhart-Macdonald Reciprocity Law.

\end{document}